\documentclass{amsart}
\numberwithin{equation}{section}
\newtheorem{theorem}{Theorem}[section]
\newtheorem{corollary}[theorem]{Corollary}
\newtheorem{proposition}[theorem]{Proposition}
\newtheorem{conjecture}[theorem]{Conjecture}
\newtheorem{lemma}[theorem]{Lemma}
\theoremstyle{definition}
\newtheorem{definition}[theorem]{Definition}

\newtheorem{example}[theorem]{Example}

\theoremstyle{remark}
\newtheorem{remark}[theorem]{\bf\em Remark}

\usepackage{graphicx}
\usepackage{amsfonts}
\usepackage{amssymb}
\usepackage{amsmath}
\usepackage{color}
\usepackage[inner=1in,outer=1in,bottom=1in,top=1in]{geometry}

\newcommand{\leg}[2]{\left(\frac{#1}{#2}\right)}

\makeatletter
\def\Ddots{\mathinner{\mkern1mu\raise\p@
\vbox{\kern7\p@\hbox{.}}\mkern2mu
\raise4\p@\hbox{.}\mkern2mu\raise7\p@\hbox{.}\mkern1mu}}
\makeatother

\begin{document}
\title[Recursions and trapezoid, symmetric and rotation symmetric functions]{Recursions associated to trapezoid, symmetric and rotation symmetric functions over Galois fields}

\author{Francis N. Castro}
\address{Department of Mathematics, University of Puerto Rico, San Juan, PR 00931}
\email{franciscastr@gmail.com}

\author{Robin Chapman}
\address{Department of Mathematics, University of Exeter, Exeter, EX4 4QF, UK}
\email{r.j.chapman@exeter.ac.uk}

\author{Luis A. Medina}
\address{Department of Mathematics, University of Puerto Rico, San Juan, PR 00931}
\email{luis.medina17@upr.edu}

\author{L. Brehsner Sep\'ulveda}
\address{Department of Mathematics, University of Puerto Rico, San Juan, PR 00931}
\email{leonid.sepulveda1@upr.edu}

\begin{abstract}
Rotation symmetric Boolean functions are invariant under circular translation of indices.  These functions have very rich cryptographic properties and have been used in different cryptosystems.
Recently, Thomas Cusick proved that exponential sums of rotation symmetric Boolean functions satisfy homogeneous linear recurrences with integer coefficients.  In this work, a generalization of this 
result is proved over any Galois field.  That is, exponential sums over Galois fields of rotation symmetric polynomials satisfy linear recurrences with integer coefficients.
 In the particular case of $\mathbb{F}_2$, an elementary method is used to obtain explicit recurrences for exponential sums of some of these functions.  The concept of trapezoid Boolean function is also introduced and it is showed that the linear recurrences that exponential sums of trapezoid Boolean functions satisfy are the same as the ones satisfied by exponential sums of 
the corresponding rotations symmetric Boolean functions.  Finally, it is proved that exponential sums of trapezoid and symmetric polynomials also satisfy linear recurrences with integer coefficients over any 
Galois field $\mathbb{F}_q$.  Moreover, the Discrete Fourier Transform matrix and some Complex Hadamard matrices appear as examples in some of our explicit formulas of these recurrences.
\end{abstract}

\date{\today}

\maketitle
\section{Introduction}
A Boolean function is a function from the vector space $\mathbb{F}_2^n$ to  $\mathbb{F}_2$.  Boolean functions are part of a beautiful branch of combinatorics with applications to many scientific areas.  
Some particular examples are the areas of theory of error-correcting codes and cryptography.   Efficient cryptographic implementations of Boolean functions with many variables is a challenging problem due to memory 
restrictions of current technology. Because of this, symmetric Boolean functions are good candidates for efficient implementations. However, symmetry is a too special property and may imply that these 
implementations are vulnerable to attacks.

In \cite{piequ}, Pieprzyk and Qu introduced rotation symmetric Boolean functions. A {\it rotation symmetric Boolean function} in $n$ variables is a function which is invariant under the action of the cyclic group $C_n$ on the set $\mathbb{F}_2^n$.   For example, let $X_i\in \mathbb{F}_2$ for $1\leq i\leq n$.  Define, for $1\leq k \leq n$, the shift function
\begin{equation}
E_n^k(X_i) = \begin{cases}
 X_{i+k} & \text{ if }i+k\leq n, \\
 X_{i+k-n} & \text{ if }i+k>n.
\end{cases}
\end{equation}
Extend this definition to $\mathbb{F}_2^n$ by defining
\begin{equation}
E_n^k(X_1,X_2,\cdots, X_n) = (E_n^k(X_1),E_n^k(X_2),\cdots, E_n^k(X_n)).
\end{equation}
The shift function $E_n^k$ can also be extended to monomials via
\begin{equation}
E_n^k(X_{i_1}X_{i_2}\cdots X_{i_t}) = E_n^k(X_{i_1})E_n^k(X_{i_2}),\cdots E_n^k(X_{i_t}).
\end{equation}
A Boolean function $F({\bf X})$ in $n$ variables is a rotation symmetric Boolean function if and only if for any $(X_1\cdots, X_n)\in\mathbb{F}_2^n$, 
\begin{equation}
F(E_n^k(X_1\cdots, X_n))=F(X_1\cdots, X_n)
\end{equation}
for every $1\leq k\leq n$.
Pieprzyk and Qu showed that these functions are useful in the design of fast hashing algorithms with strong cryptographic properties.  This work sparked interest in these functions and today their study is an active area of research \cite{BCP,cusickjohns, cusickstanica, dalaimaitrasarkar, hell, maxhellmaitra, stanicamaitra, stanicamaitraclark}.

Every Boolean function in $n$ variables can be identified with a multi-variable Boolean polynomial.  This polynomial is known as the algebraic normal form (ANF for short) of the Boolean function.  The 
degree of a Boolean function $F({\bf X})$ is the degree of its ANF.  The ANF of a rotation symmetric Boolean function is very well-structured.  For example, suppose we have a rotation symmetric Boolean function in 5 variables.  Suppose that $X_1X_2X_3$ is part of the ANF of the function.  Then, the terms
\begin{eqnarray}
E_5^1(X_1X_2X_3)&=&X_2X_3X_4\\\nonumber
E_5^2(X_1X_2X_3)&=&X_3X_4X_5\\\nonumber 
E_5^3(X_1X_2X_3)&=&X_4X_5X_1\\ \nonumber
E_5^4(X_1X_2X_3)&=&X_5X_1X_2
\end{eqnarray}
are also part of its ANF.  Similarly, suppose that $X_1X_3$ is also a term of the ANF.  Then, 
$$X_2X_4, X_3X_5, X_4X_1, X_5X_2$$
are also part of the ANF.  An example of a rotation symmetric Boolean function with this property is given by
\begin{eqnarray}
\label{rotex}
R({\bf X}) &=&X_1X_2X_3+X_2X_3X_4+X_3X_4X_5+X_4X_5X_1+ X_5X_1X_2+\\\nonumber
&&X_1X_3+X_2X_4+ X_3X_5+ X_4X_1+ X_5X_2.
\end{eqnarray}
Therefore, once a monomial $X_{i_1}\cdots X_{i_t}$ is part of the ANF of a rotation symmetric Boolean function, so is $E_n^k(X_{i_1}\cdots X_{i_t})$ for all $1\leq k \leq n$.  This implies that the 
information encoded in the ANF of a rotation symmetric Boolean function can be obtained with minimal information.  Define the set
\begin{equation}
RSet_n(X_{i_1}\cdots X_{i_t})=\{E_n^{k}(X_{i_1}\cdots X_{i_t})\,|\, 1\leq k \leq n\}.
\end{equation}
For example,
\begin{equation}
RSet_5(X_1X_2X_3)=\{X_2X_3X_4, X_3X_4X_5, X_4X_5X_1, X_5X_1X_2, X_1X_2X_3\}.
\end{equation}
Select as a representative for the set $RSet_n(X_{i_1}\cdots X_{i_t})$ the first element in the lexicographic order.  For example, the representative for 
$$\{X_2X_3X_4, X_3X_4X_5, X_4X_5X_1, X_5X_1X_2, X_1X_2X_3\}$$
is $X_1X_2X_3$.  Observe that if the rotation symmetric Boolean function is not constant, then $X_1$ always appears in the lexicographically first element of $RSet_n(X_{i_1}\cdots X_{i_t})$.
The {\it short algebraic normal form} (or SANF) of a rotation symmetric Boolean function is a function of the form
\begin{equation}
a_0 +a_1 X_1+\sum a_{1,j}X_1X_j+\cdots+a_{1,2,\cdots,n} X_1X_2\cdots X_n,
\end{equation}
where $a_0,a_1, a_{1,j},\cdots, a_{1,2,\cdots,n} \in \mathbb{F}_2$ and the existence of the term $X_1X_{i_2}\cdots X_{i_t}$ implies the existence of every term in $$RSet_n(X_1X_{i_2}\cdots X_{i_t})$$
in the ANF.  For example, the SANF of the rotation symmetric Boolean function (\ref{rotex}) is given by
\begin{equation}
 X_1X_3+X_1X_2X_3.
\end{equation}

Let $1<j_1<\cdots< j_s$ be integers.  A rotation symmetric Boolean function of the form 
\begin{equation}
\label{monrot}
R_{j_1,\cdots, j_s}(n)=X_1X_{j_1}\cdots X_{j_s}+X_2X_{j_1+1}\cdots X_{j_s+1}+\cdots+X_nX_{j_1-1}\cdots X_{j_s-1},
\end{equation}
where the indices are taken modulo $n$ and the complete system of residues is $\{1,2,\cdots, n\}$, is called a {\it monomial rotation symmetric} Boolean function.  For example, the rotation symmetric 
Boolean function (\ref{rotex}) is given by
\begin{equation}
 R({\bf X}) = R_{2,3}(5) + R_{3}(5).
\end{equation}
Sometimes the notation $(1,j_1,\cdots, j_s)_n$ is used to represent the monomial rotation Boolean function (\ref{monrot}), see \cite{cusickArXiv}.

In some applications related to cryptography it is important for Boolean functions to be balanced.  A balanced Boolean function is one for which the number of zeros and the number 
of ones are equal in its truth table.  Balancedness of Boolean functions can be studied from the point of view of exponential sums.  The {\it exponential sum} of an $n$-variable Boolean function $F({\bf X})$ is defined as
\begin{equation}
 S(F)=\sum_{{\bf x}\in \mathbb{F}_2^n} (-1)^{F({\bf x})}.
\end{equation}
Observe that a Boolean function $F({\bf X})$ is balanced if and only if $S(F)=0$.  This gives importance to the study of exponential sums.  This point of view is also a very active area of research. For 
some examples, please refer to \cite{sperber, ax, cm1, cm2, cm3, cusick4, fspectrum, mm1, mm, fdegree}.

Let $F({\bf X})$ be a Boolean function.  List the elements of $\mathbb{F}_2^ n$ in lexicographic order and label them as ${\bf x}_0=(0,0,\cdots, 0)$, ${\bf x}_{1}=(0,0,\cdots, 1)$ and so on.  The vector $(F({\bf x}_0),F({\bf x}_1),\cdots, F({\bf x}_{2^n-1}))$ is called the {\it truth table} or $F$.  The {\it Hamming weight} of $F$, denoted by $\text{wt}(F)$, is the number of 1's in the truth table of $F$.  Observe that a Boolean function in $n$ variables is balanced if and only if its Hamming weight is $2^{n-1}$.  The Hamming weight of a Boolean function $F$ and its exponential sums are related by the equation
\begin{equation}
\label{weightexpsum}
\text{wt}(F)=\frac{2^n-S(F)}{2}.
\end{equation}

The study of weights of rotations symmetric Boolean functions has received some attention lately \cite{BCP, cusickjohns, cusickstanica,stanicamaitra}.  In particular, it has been observed that weights of cubic  rotation symmetric Boolean functions are linear recursive with constant coefficients \cite{BCP, cusickjohns}.  Recently, Cusick \cite{cusickArXiv} showed that weights of any rotation symmetric 
Boolean function satisfy linear recurrences with integer coefficients.  Since the exponential sum and the weight function of a Boolean function are related by (\ref{weightexpsum}), then it is also true 
that exponential sums of rotation symmetric Boolean functions satisfy linear recurrences with integer coefficients.

One of the most important results in this work is a generalization of Cusick's Theorem over any Galois field.  To be specific, let $q=p^r$ with $p$ prime and $r\geq1$. Exponential sums over $\mathbb{F}_q$ 
of monomial rotation symmetric polynomials (and linear combinations of them) satisfy homogeneous linear recurrences with integer coefficients.  Remarkably, this can be proved by elementary means.  
Another important result included in this work is that exponential sums over $\mathbb{F}_q$ of elementary symmetric polynomials and linear combinations of them also satisfy linear recurrences with integer 
coefficients.  Surprisingly, the Discrete Fourier Transform matrix, some Complex Hadamard matrices and the quadratic Gauss sum mod $p$ appear in the study of the recurrences considered in this work.

This article is divided as follows. The next section is an introduction of the elementary method used to obtain the recurrences.  This introduction is done over $\mathbb{F}_2$ in order to solidify the 
intuition.  The reader interested in the generalization is invited to skip this section, however, the reader is encouraged to read the definition of trapezoid functions, as they are used through 
out the article. In section \ref{anyGalois} linear recurrences with integer coefficients are obtained for exponential sums trapezoid functions over Galois fields.  Moreover, it is in this section where 
it is proved that exponential sums over $\mathbb{F}_q$ of monomial rotation symmetric polynomials and linear combinations of them satisfy linear recurrences with integer coefficients.  The same technique 
is used in the section \ref{symmetriccase} to prove that exponential sums over $\mathbb{F}_q$ of elementary symmetric polynomials and linear combinations of them also satisfy linear recurrences with 
integer coefficients.  Finally, in the last section, some conjectures about the initial conditions of some of the sequences considered in this work are presented.


\section{Linear recurrences over $\mathbb{F}_2$}

As mentioned in the introduction, Cusick \cite{cusickArXiv} recently showed that exponential sums of rotation symmetric Boolean functions satisfy homogeneous linear recurrences with integer coefficients.  
This fact was suggested by some previous works on the subject. For example, in \cite{cusickstanica}, Cusick and St$\check{\mbox{a}}$nic$\check{\mbox{a}}$ provided a linear recursion for the sequence of
weights for the monomial rotation function $(1,2,3)_n$.  This recursion, however, was not homogeneous, but it could be transformed into a homogeneous one, see \cite{BCP}.
Later, Cusick and Johns \cite{cusickjohns} provided recursions for  weights of cubic rotation symmetric Boolean functions.

In this section we use elementary machinery to provide explicit homogeneous linear recurrences with integer coefficients for exponential sums of some rotation symmetric Boolean functions.  The idea is to
show that exponential sums of rotation symmetric Boolean functions satisfy the same linear recurrences of exponential sums of trapezoid Boolean functions (see definition below).  We prove this fact
using elementary machinery and, at this early stage, without the use linear algebra.  In the next section we show that exponential sums of rotation symmetric functions over any Galois field satisfy 
linear recurrences.  The reader interested in this generalization may skip this section, but not before reading the definition of trapezoid functions.

Define the {\it trapezoid} Boolean function in $n$ variables of degree $k$ as
\begin{equation}
 \tau_{n,k} = \sum _{j=1}^{n-k+1} X_jX_{j+1}\cdots X_{j+k-1}.
\end{equation}
For example,
\begin{eqnarray*}
 \tau_{7,3}&=& X_1 X_2 X_3+X_2 X_3 X_4+X_3 X_4 X_5+X_4 X_5 X_6+X_5 X_6 X_7\\
 \tau_{6,4}&=& X_1 X_2 X_3 X_4+X_2 X_3 X_4 X_5+X_3 X_4 X_5 X_6.
\end{eqnarray*}
The name trapezoid comes from counting the number of times each variable appears in the function $\tau_{n,k}$.  For example, consider $\tau_{7,3}$.  Observe that $X_1$ appears 1 time in $\tau_{7,3}$, $X_2$ 
appears 2 times, $X_3$, $X_4$ and $X_5$ appears 3 times each, $X_6$ appears twice, and $X_7$ appears once.  Plotting the these values and connecting the dots produces the shape of an isosceles trapezoid. 
Figure \ref{trap73} is a graphical representation of this.  The Boolean variable $X_i$ is represented by $i$ in the $x$-axis. The $y$-axis corresponds to the number of times the variable appears in $\tau_{7,3}$.
\begin{figure}[h!]
\centering
\includegraphics[width=3in]{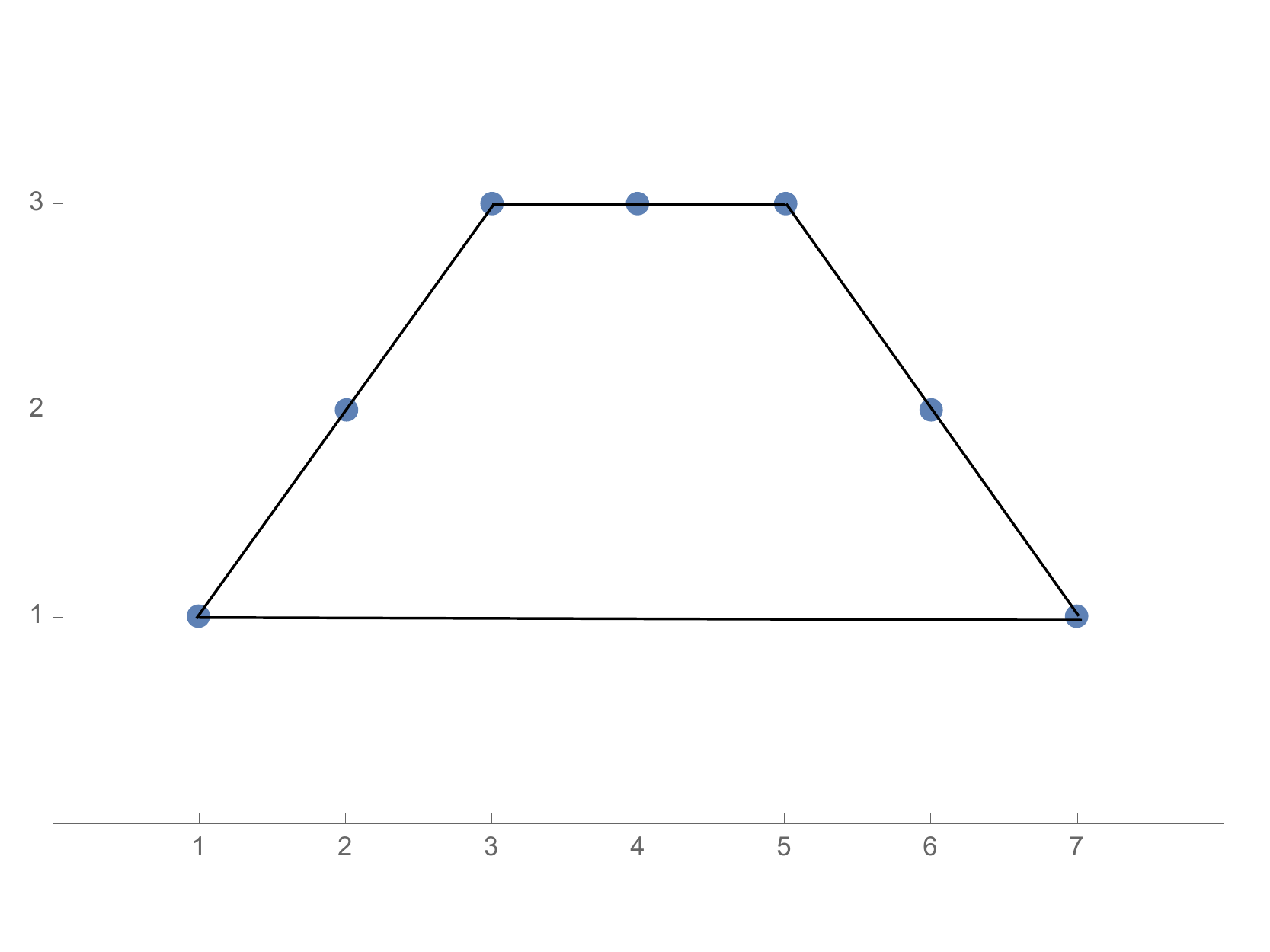}
\caption{Trapezoid associated to the Boolean function $\tau_{7,3}$}
\label{trap73}
\end{figure}

The opposite is also true, that is, for every isosceles trapezoid that can be constructed by steps of length at most 1, one can construct a trapezoid Boolean function.

It turns out that sequences of exponential sums of trapezoid Boolean functions of fixed degree satisfy homogeneous linear recurrences with integer coefficients.  These linear recurrences are the same 
satisfied by sequences of exponential sums of $(1,2,\cdots,k)$-rotation symmetric Boolean functions.  Remarkably, this fact can be proved by elementary means by ``playing" a simple game of turning 
{\it ON} and {\it OFF} some of the variables.   Given a Boolean variable $X_i$, we say that it is turned {\it OFF} if $X_i$ assumes the value 0 and turned {\it ON} if the variable assumes the value 1.
In other words, each Boolean variable represents a ``switch" with two options: 0 ({\it OFF}) and 1 ({\it ON}).

We start the discussion with the recurrence for exponential sums of trapezoid Boolean functions.

\begin{theorem}
\label{trapezoidTHM}
 The sequence $\{S(\tau_{n,k})\}_{n=k}^\infty$ satisfies a homogeneous linear recurrence with integer coefficients whose characteristic polynomial is given by 
 \begin{equation}
 \label{charpoly}
  p_k(X)=X^k -2 (X^{k-2}+X^{k-3}+\cdots+X+1).
 \end{equation}
\end{theorem}
\begin{proof}
For the sake of simplicity, we present, in detail, the proof for the cases $k=3$ and $k=4$.  The general case becomes clear after that.  Moreover, the complete proof of a generalization of this theorem over any Galois field is presented in section \ref{anyGalois}.
 
Start with the case $k=3$.  Observe that by turning $X_n$ {\it OFF} and {\it ON} we get the identity
 \begin{equation}
 \label{firsttau}
  S(\tau_{n,3}) = S(\tau_{n-1,3})+S(\tau_{n-1,3}+X_{n-2}X_{n-1}).
 \end{equation}
Consider now $S(\tau_{n-1,3}+X_{n-2}X_{n-1})$.  Turn $X_{n-1}$ {\it OFF} and {\it ON} to get
\begin{equation}
\label{secondtau}
 S(\tau_{n-1,3}+X_{n-2}X_{n-1})=S(\tau_{n-2,3})+S(\tau_{n-2,3}+X_{n-2}+X_{n-3}X_{n-2}).
\end{equation}
Finally, turn $X_{n-2}$ {\it OFF} and {\it ON} to get
\begin{equation}
 S(\tau_{n-2,3}+X_{n-2}+X_{n-3}X_{n-2})=S(\tau_{n-3,3})-S(\tau_{n-3,3}+X_{n-3}+X_{n-4}X_{n-3}).
\end{equation}
The last equation is equivalent (after relabeling) to 
\begin{equation}
\label{thirdtau}
S(\tau_{n,3})=S(\tau_{n+1,3}+X_{n+1}+X_{n}X_{n+1})+S(\tau_{n,3}+X_{n}+X_{n-1}X_{n}).
\end{equation}

Observe that equations (\ref{firsttau}) and (\ref{secondtau}) can be combined to obtain
\begin{equation}
\label{fourthtau}
S(\tau_{n,3})=S(\tau_{n-1,3})+S(\tau_{n-2,3})+ S(\tau_{n-2,3}+X_{n-2}+X_{n-3}X_{n-2}).
\end{equation}
Let $a_{n,3} =S(\tau_{n,3}+X_{n}+X_{n-1}X_{n})$.  Note that (\ref{thirdtau}) implies that $S(\tau_{n,3})=a_{n+1,3}+a_{n,3}$.  Therefore, (\ref{fourthtau}) can be re-written as
\begin{eqnarray}
(a_{n+1,3}+a_{n,3}) = (a_{n,3}+a_{n-1,3})+(a_{n-1,3}+a_{n-2,3})+a_{n-2,3},
\end{eqnarray}
which is equivalent to 
\begin{equation}
a_{n+1,3} = 2a_{n-1,3}+2a_{n-2,3}.
\end{equation}
This implies that $\{a_{n,3}\}$ satisfies the linear recurrence whose characteristic polynomial is given by $p_3(X)$.  Since $S(\tau_{n,3})=a_{n+1,3}+a_{n,3}$, then $\{S(\tau_{n,3})\}$ also satisfies such recurrence and the result holds for $k=3$.

Consider now the case when $k=4$.  As it was done in the case when $k=3$,  turning {\it OFF} and {\it ON} several variables leads to 
\begin{eqnarray}
\label{fifthtau}
S(\tau_{n,4})&=&S(\tau_{n-1,4})+S(\tau_{n-2,4})+S(\tau_{n-3,4})\\\nonumber
&&+S(\tau_{n-3,4}+X_{n-3}+X_{n-4}X_{n-3}+X_{n-5}X_{n-4}X_{n-3})
\end{eqnarray}
and
\begin{eqnarray}
S(\tau_{n,4})&=& S(\tau_{n+1,4}+X_{n+1}+X_{n}X_{n+1}+X_{n-1}X_{n}X_{n+1})\\\nonumber
&&+S(\tau_{n,4}+X_{n}+X_{n-1}X_{n}+X_{n-2}X_{n-1}X_{n}).
\end{eqnarray}
Now let $a_{n,4}=S(\tau_{n,4}+X_{n}+X_{n-1}X_{n}+X_{n-2}X_{n-1}X_{n})$ and observe that (\ref{fifthtau}) can be re-written as
\begin{equation}
(a_{n+1,4}+a_{n,4})=(a_{n,4}+a_{n-1,4})+(a_{n-1,4}+a_{n-2,4})+(a_{n-2,4}+a_{n-3,4})+a_{n-3,4},
\end{equation}
which is equivalent to
\begin{equation}
a_{n+1,4} = 2a_{n-1,4}+2a_{n-2,4}+2a_{n-3,4}.
\end{equation}
Therefore, $\{a_{n,4}\}$ satisfies the linear recurrence whose characteristic polynomial is given by $p_4(X)$.  Since $S(\tau_{n,4})=a_{n+1,4}+a_{n,4}$, then $\{S(\tau_{n,4})\}$ also satisfies such 
recurrence and the result also holds for $k=4$.  

In general, $S(\tau_{n,k})$ can be expressed as
\begin{align}
S(\tau_{n,k})&=\sum_{i=1}^{k-1} S(\tau_{n-i,k})+S\left(\tau_{n-k+1,k}+\sum_{j=0}^{k-2} \prod_{i=0}^j X_{n-k+1-i}\right)
\end{align}
and as
\begin{align}
S(\tau_{n,k})&=S\left(\tau_{n+1,k}+\sum_{j=0}^{k-2} \prod_{i=0}^j X_{n+1-i}\right)+S\left(\tau_{n,k}+\sum_{j=0}^{k-2} \prod_{i=0}^j X_{n-i}\right).
\end{align}
Combine these equations and proceed as before to obtain the result.  This concludes the proof.
\end{proof}

It turns out that the sequence of exponential sums of $(1,2,\cdots,k)$-rotation symmetric Boolean functions, that is, of $R_{2,3,\cdots, k}(n)$, also satisfies the linear recurrence whose characteristic 
polynomial is given $p_k(X)$. This is a well-known result for the case when $k=3$ (\cite{BCP, cusickjohns}), but, to the knowledge of the authors, the closed formula for the general case is new. 
Before proving that $\{S(R_{2,3,\cdots, k}(n))\}$ satisfies the linear recurrence with characteristic polynomial $p_k(X)$, we show an auxiliary result which can be proved using the same arguments as in the 
proof of Theorem \ref{trapezoidTHM}.

\begin{lemma}
\label{trapezoidplusLemma}
 Let $\tau_{n,k}$ be the trapezoid Boolean function of degree $k$ in $n$ variables.  Suppose that $F({\bf X})$ is a Boolean polynomial in the first $j$ variables with $j<k$.  Then, the sequences
 $$\{S(\tau_{n,k}+F({\bf X}))\}$$ and $$\{S(\tau_{n,k}+F({\bf X})+X_n+X_nX_{n-1}+X_nX_{n-1}X_{n-2}+\cdots+X_nX_{n-1}\cdots X_{n-k+2})\}$$ satisfies the linear recurrence whose characteristic polynomial 
 is given by $p_k(X)$.
\end{lemma}

\begin{proof}
 The proof of this result follows the same argument of the proof of Theorem \ref{trapezoidTHM}.   
\end{proof}

Theorem \ref{trapezoidTHM} and Lemma \ref{trapezoidplusLemma} is all that is needed to show that the sequence of exponential sums of $(1,2,\cdots,k)$-rotation symmetric Boolean functions satisfies the linear
recurrence with characteristic polynomial $p_k(X)$.

\begin{theorem}
 The sequence $\{S(R_{2,3,\cdots, k}(n))\}$ satisfies the homogeneous linear recurrence whose characteristic polynomial is given by $p_k(X)$. 
\end{theorem}

\begin{proof}
 This result can also be proved by turning {\it OFF} and {\it ON} several variables.  As before, we provide the proof for the case when $k=4$.  The general case follows the same argument.
 
 To start the argument, turn {\it OFF} and {\it ON} the variable $X_n$ to get
 \begin{equation}
 \label{firstreduction}
  S(R_{2,3,4}(n)) = S(\tau_{n-1,4})+S(\tau_{n-1,4}+X_1X_2X_3+X_1X_2X_{n-1}+X_1X_{n-2}X_{n-1}).
 \end{equation}
Consider the second term of the right hand side of this equation.  Turn $X_{n-1}$ {\it OFF} and {\it ON} to get
\begin{align}
\label{secondterm2}
 S(\tau_{n-1,4}&+X_1X_2X_3+X_1X_2X_{n-1}+X_1X_{n-2}X_{n-1})\\\nonumber
 &=S(\tau_{n-2,4}+X_1X_2X_3)\\\nonumber
 &\,\,\,+S(\tau_{n-2,4}+X_1X_2+X_1X_2X_3+X_1X_{n-2}+X_{n-3}X_{n-2}+X_{n-4}X_{n-3}X_{n-2}).
\end{align}
Again, consider the second term of the right hand side of equation (\ref{secondterm2}). Turn $X_{n-2}$ {\it OFF} and {\it ON} to get
\begin{align}
\label{secondterm3}
 S(\tau_{n-2,4}&+X_1X_2+X_1X_2X_3+X_1X_{n-2}+X_{n-3}X_{n-2}+X_{n-4}X_{n-3}X_{n-2})\\\nonumber
 &=S(\tau_{n-3,4}+X_1X_2+X_1X_2X_3)\\\nonumber
 &\,\,\,+S(\tau_{n-3,4}+X_1+X_1X_2+X_1X_2X_3+X_{n-3}+X_{n-4}X_{n-3}+X_{n-5}X_{n-4}X_{n-3}).
\end{align}
Equations (\ref{firstreduction}), (\ref{secondterm2}) and (\ref{secondterm3}) lead to the equation
 \begin{align}
  S(R_{2,3,4}(n)) &= S(\tau_{n-1,4})+S(\tau_{n-2,4}+X_1X_2X_3)+S(\tau_{n-3,4}+X_1X_2+X_1X_2X_3)\\\nonumber
  & \,\,\,+S(\tau_{n-3,4}+X_1+X_1X_2+X_1X_2X_3+X_{n-3}+X_{n-4}X_{n-3}+X_{n-5}X_{n-4}X_{n-3}).
 \end{align}
Theorem \ref{trapezoidTHM} and Lemma \ref{trapezoidplusLemma} imply that $\{S(\tau_{n-1,4})\}$,  $\{S(\tau_{n-2,4}+X_1X_2X_3)\}$, $\{S(\tau_{n-3,4}+X_1X_2+X_1X_2X_3)\}$ and
$$\{S(\tau_{n-3,4}+X_1+X_1X_2+X_1X_2X_3+X_{n-3}+X_{n-4}X_{n-3}+X_{n-5}X_{n-4}X_{n-3})\}$$
satisfy the linear recurrence whose characteristic polynomial $p_4(X)$.  Since $\{S(R_{2,3,4}(n))\}$ is a linear combination of them, then the result holds when $k=4$.

In general, $S(R_{2,3,\cdots, k}(n))$ can be expressed as
\begin{align}
S(R_{2,3,\cdots, k}(n))&=S(\tau_{n-1,k})+\sum_{m=0}^{k-3} S\left(\tau_{n-2-m,k}+\sum_{j=0}^m \prod_{i=1}^{k-1-j} X_i\right)\\\nonumber
&\,\,\,+S\left(\tau_{n-k+1,k}+\sum_{j=1}^{k-1} \left(\prod_{i=1}^{j} X_i+ \prod_{i=0}^{j-1} X_{n-k+1-i}\right)\right)
\end{align}
Invoke Theorem \ref{trapezoidTHM} and Lemma \ref{trapezoidplusLemma} to get the result.  This concludes the proof.
\end{proof}

The same technique can be applied to find linear recurrences of exponential sums other rotations.  Recall that 
\begin{equation}
R_{j_1,\cdots, j_s}(n)=X_1X_{j_1}\cdots X_{j_s}+X_2X_{j_1+1}\cdots X_{j_s+1}+\cdots+X_nX_{j_1-1}\cdots X_{j_s-1},
\end{equation}
where the indices are taken modulo $n$ and the complete system of residues is $\{1,2,\cdots, n\}$. 
We define the equivalent of the trapezoid Boolean function for $R_{j_1,\cdots, j_s}(n)$ as
\begin{equation}
T_{j_1,\cdots,j_s}(n)=X_1X_{j_1}\cdots X_{j_s}+X_2X_{j_1+1}\cdots X_{j_s+1}+\cdots+X_{n+1-j_s}X_{j_1+n-j_s}\cdots X_{j_{s-1}+n-j_s}X_n.
\end{equation}
For instance, under this notation one has
\begin{equation}
\tau_{n,k}=T_{2,3,\cdots, k}(n).
\end{equation}
It turns out that for $k\geq 4$, the sequences $\{S(R_{2,3,\cdots,k-2,k}(n))\}$ and $\{S(R_{2,3,\cdots,k-2,k+1}(n))\}$ both satisfy the linear recurrence whose characteristic polynomial is
\begin{equation}
q_k(X)=X^{k+1}-2X^{k-1}-2X^{k-2}-\cdots-2X^3-4.
\end{equation}
As just mentioned, this can be proved by playing a game of turning {\it ON} and {\it OFF} some variables.  However, the process becomes somewhat tedious at a very early stage.

For example, recall that Theorem \ref{trapezoidTHM} is an auxiliary result that was used to show that $\{S(R_{2,3,\cdots, k}(n))\}$ satisfies the linear recurrence with characteristic polynomial  $p_k(X)$.  Let us show the equivalent of Theorem \ref{trapezoidTHM} for $\{S(R_{2,4}(n))\}$.  The idea is to show the reader how tedious the process can get.  Recall that the equivalent of the trapezoid Boolean function for this problem is 
\begin{equation}
T_{2,4}(n)=X_1X_2X_4+X_2X_3X_5+\cdots+X_{n-3}X_{n-2}X_n.
\end{equation}
Start with the equation
\begin{eqnarray}
\label{gentrap1}
 S(T_{2,4}(n)+X_{n-1}X_n)&=&S(T_{2,4}(n+1)+X_{n-1}X_n+X_{n+1}+X_{n}X_{n+1})+\\\nonumber
 && S(T_{2,4}(n)+X_{n-2}X_{n-1}+X_{n}+X_{n-1}X_{n}),
\end{eqnarray}
which is a consequence of turning {\it OFF} and {\it ON} the variable $X_{n+1}$. On the other hand, by turning $X_{n}$ {\it OFF} and {\it ON} one gets
\begin{eqnarray}
\label{gentrap2}
S(T_{2,4}(n)+X_{n-1}X_n)&=&S(T_{2,4}(n-1)+ S(T_{2,4}(n-1)+X_{n-1}+X_{n-2}X_{n-3}).
\end{eqnarray}
This gave us two equations for $S(T_{2,4}(n)+X_{n-1}X_n)$.

Consider now the right hand side of (\ref{gentrap2}).  Turn $X_{n-1}$ {\it OFF} and {\it ON} to get 
\begin{eqnarray}
\label{gentrap3}
 S(T_{2,4}(n-1)+X_{n-1}+X_{n-2}X_{n-3})&=& S(T_{2,4}(n-2)+X_{n-2}X_{n-3})-\\\nonumber
 &&S(T_{2,4}(n-2)+X_{n-4}X_{n-3}+X_{n-3}X_{n-2})
\end{eqnarray}
Now turn $X_{n-2}$ {\it OFF} and {\it ON} to get the equation
\begin{eqnarray}\nonumber
\label{gentrap4}
 S(T_{2,4}(n-2)+X_{n-4}X_{n-3}+X_{n-3}X_{n-2})&=&S(T_{2,4}(n-3)+X_{n-4}X_{n-3})+\\
 &&S(T_{2,4}(n-3)+X_{n-5}X_{n-4}+X_{n-3}+X_{n-4}X_{n-3}).
\end{eqnarray}
Combine equations (\ref{gentrap1}), (\ref{gentrap2}), (\ref{gentrap3}), and (\ref{gentrap4}) to get
\begin{eqnarray}\nonumber
\label{gentrap5}
S(T_{2,4}(n)&=&S(T_{2,4}(n+1)+X_{n}X_{n+1})-S(T_{2,4}(n-1)+X_{n-2}X_{n-1})+S(T_{2,4}(n-2)+X_{n-3}X_{n-2})\\
&&+S(T_{2,4}(n-2)+X_{n-4}X_{n-3}+X_{n-2}+X_{n-3}X_{n-2}).
\end{eqnarray}

Now let $a_n=S(T_{2,4}(n)+X_{n-2}X_{n-1}+X_{n}+X_{n-1}X_{n})$.  Observe that equation (\ref{gentrap1}) can be re-written as
\begin{equation}
 S(T_{2,4}(n)+X_{n-1}X_n)=a_{n+1}+a_n.
\end{equation}
This and equation (\ref{gentrap5})  imply
\begin{eqnarray}
 S(T_{2,4}(n))&=&(a_{n+2}+a_{n+1})-(a_n+a_{n-1})+(a_{n-1}+a_{n-2})+a_{n-2}\\\nonumber
 &=&a_{n+2}+a_{n+1}-a_n+2a_{n-2}.
\end{eqnarray}
On the other hand, by switching {\it OFF} and {\it ON} several variables one obtains
\begin{eqnarray}
 S(T_{2,4}(n))&=& S(T_{2,4}(n-1))+S(T_{2,4}(n-2)+X_{n-3}X_{n-2})+S(T_{2,4}(n-3)+X_{n-4}X_{n-3})\\\nonumber
 &&+ S(T_{2,4}(n-3)+X_{n-5}X_{n-4}+X_{n-3}+X_{n-4}X_{n-3}).
\end{eqnarray}
Writing this last equation in terms of $a_n$ one gets
\begin{eqnarray}
 (a_{n+2}+a_{n+1}-a_n+2a_{n-2})&=&(a_{n+1}+a_{n}-a_{n-1}+2a_{n-3})\\\nonumber
 &&+(a_{n-1}+a_{n-2})+(a_{n-2}+a_{n-3})+a_{n-3},
\end{eqnarray}
which simplifies to
\begin{equation}
 a_{n+2}=2a_n+4a_{n-3}.
\end{equation}
The characteristic polynomial for this recurrence is $q_3(X)$.

Other examples on which this elementary method can be used to find explicit formulas for linear recurrences include the sequence
\begin{equation}
\{S(R_{2,3,\cdots,k}(n)+R_{2,3,\cdots,k-1}(n))\},
\end{equation}
which satisfies the linear recurrence with characteristic polynomial 
\begin{equation}
x^k-2x^{k-1}+2,
\end{equation}
the sequence
\begin{equation}
\{S(R_{2,3,\cdots, k-1,k}(n)+R_{2,3,\cdots,k-2,k}(n))\},
\end{equation}
which satisfies the linear recurrence with characteristic polynomial 
\begin{equation}
x^k-2x^{k-1}+2x-2,
\end{equation}
and the sequence 
\begin{equation}
\{S(R_{2, 3,\cdots, k-2,k}(n)+ R_{2, 3, \cdots, k-1}(n)+R_{2, 3, \cdots,k}(n))\},
\end{equation}
which satisfies the linear recurrence with characteristic polynomial 
\begin{equation}
x^k-2(x^{k-2}+x^{k-3}+\cdots+x^2+1).
\end{equation}
However, the process is somewhat tedious to be done by hand.  Automatization seems to be the way to go.  The reader is invited to read Cusick's work \cite{cusickArXiv}, which
includes a {\em Mathematica} code that calculates a linear recurrences for the weights of a given rotation.


\section{Linear recurrences over $\mathbb{F}_q$}
\label{anyGalois}

In this section we show that exponential sums of rotation functions over Galois fields satisfy linear recurrence.  This is a generalization of Cusick's result.

Consider the Galois field $\mathbb{F}_q = \{0,\alpha_1,\cdots,\alpha_{q-1}\}$ where $q=p^r$ with $p$ prime and $r\geq 1$.  Recall that the exponential sum 
of a function $F:\mathbb{F}_q^n \to \mathbb{F}_q$ is given by
\begin{equation}
 S_{\mathbb{F}_q}(F)=\sum_{{\bf x}\in \mathbb{F}^n_q} e^{\frac{2\pi i}{p} \text{Tr}_{\mathbb{F}_q/\mathbb{F}_p}(F({\bf x}))},
\end{equation}
where $\text{Tr}_{\mathbb{F}_q/\mathbb{F}_p}$ represents the field trace function from $\mathbb{F}_q$ to $\mathbb{F}_p$.
The same technique used for exponential sums of Boolean functions can be used in general.  However, instead of having two options for the ``switch", we now have $q$ of them.  Let $X$ be a variable which takes values on $\mathbb{F}_q$.  As before, we say that the variable $X$ can be turned {\it OFF} or {\it ON}, however, this time the term ``turn {\it OFF}" means that $X$ assumes the value 0, while the term ``turn {\it ON}" means that $X$ assumes all values in $\mathbb{F}_q$ that are different from zero.  Think of this situation as a light switch on which you have the option to turn  {\it OFF} the light and the option to turn it {\it ON} to one of $q-1$ colors.

We consider first sequences exponential sums of trapezoid functions.  As in the case over $\mathbb{F}_2$, they satisfy linear recurrences with integer coefficients over any Galois field 
$\mathbb{F}_q$.   We start with the following lemma, which is interesting in its own right.

\begin{lemma}
\label{generallemma}
Let $k, n$ and $j$ be integers with $k>2$, $1\leq j<k$ and $n\geq k$.   Then,
\begin{equation}
S_{\mathbb{F}_q}\left(T_{2,3,\cdots,k}(n)+\sum_{s=1}^j\beta_s\prod_{l=0}^{k-s-1}X_{n-l}\right)=S_{\mathbb{F}_q}\left(T_{2,3,\cdots,k}(n)+\sum_{s=1}^j\prod_{l=0}^{k-s-1}X_{n-l}\right)
\end{equation}
for any choice of $\beta_{s} \in \mathbb{F}_q^\times$.
\end{lemma}

\begin{proof}
The proof is by induction on $n$.  Suppose first that $n=k$.  Observe that
\begin{eqnarray}\nonumber
\label{basecase}
 T_{2,3,\cdots,k}(k)+\sum_{s=1}^j\beta_s\prod_{l=0}^{k-s-1}X_{k-l}&=& X_1X_2\cdots X_k+\beta_j X_{j+1}X_{j+2}\cdots X_k+\beta_{j-1} X_{j}X_{j+1}\cdots X_k\\
 &&+\cdots+ \beta_2 X_3 X_4\cdots X_k+  \beta_1 X_2 X_3\cdots X_k.
\end{eqnarray}
Consider the right hand side of (\ref{basecase}).  If $1\leq j\leq k-2$, then make the changes of variables
\begin{eqnarray*}
 X_t &=& Y_t, \,\,\, \text{ for }j+2\leq t \leq k \\
 X_{j+1}&=& \beta_j^{-1}Y_{j+1}\\
 X_t&=& \beta_{t-1}^{-1}\beta_t Y_t, \,\,\, \text{ for }2\leq t \leq j\\ 
 X_1&=& \beta_1 Y_1.
\end{eqnarray*}
On the other hand, if $j=k-1$, then make the change of variables
\begin{eqnarray*}
 X_{k}&=& \beta_{k-1}^{-1}Y_{k}\\
 X_t&=& \beta_{t-1}^{-1}\beta_t Y_t, \,\,\, \text{ for }2\leq t \leq k-1\\ 
 X_1&=& \beta_1 Y_1.
\end{eqnarray*}
This transforms (\ref{basecase}) into 
\begin{equation}
  Y_1Y_2\cdots Y_k +\sum_{s=1}^j\prod_{l=0}^{k-s-1}Y_{k-l}.
\end{equation}
Therefore, 
\begin{equation}
 S_{\mathbb{F}_q}\left(T_{2,3,\cdots,k}(k)+\sum_{s=1}^j\beta_s\prod_{l=0}^{k-s-1}X_{k-l}\right)=S_{\mathbb{F}_q}\left(T_{2,3,\cdots,k}(k)+\sum_{s=1}^j\prod_{l=0}^{k-s-1}X_{k-l}\right).
\end{equation}
This concludes the base case.  

Suppose now that for some $n\geq k$ we have 
\begin{equation}
 S_{\mathbb{F}_q}\left(T_{2,3,\cdots,k}(n)+\sum_{s=1}^j\beta_s\prod_{l=0}^{k-s-1}X_{n-l}\right)=S_{\mathbb{F}_q}\left(T_{2,3,\cdots,k}(n)+\sum_{s=1}^j\prod_{l=0}^{k-s-1}X_{n-l}\right).
\end{equation}
Consider 
\begin{equation}
 S_{\mathbb{F}_q}\left(T_{2,3,\cdots,k}(n+1)+\sum_{s=1}^j\beta_s\prod_{l=0}^{k-s-1}X_{n+1-l}\right).
\end{equation}
Suppose first that $1\leq j \leq k-2$.  Letting $X_{n+1}$ run over every element of the field leads to 
\begin{eqnarray}\nonumber
  S_{\mathbb{F}_q}\left(T_{2,3,\cdots,k}(n+1)+\sum_{s=1}^j\beta_s\prod_{l=0}^{k-s-1}X_{n+1-l}\right)&=& S_{\mathbb{F}_q}\left(T_{2,3,\cdots,k}(n)\right)\\
  &&+\sum_{\alpha\in\mathbb{F}_q^{\times}}  S_{\mathbb{F}_q}\left(T_{2,3,\cdots,k}(n)+\sum_{s=1}^{j+1}\gamma_s(\alpha)\prod_{l=0}^{k-s-1}X_{n-l}\right),
\end{eqnarray}
where $\gamma_1(\alpha)=\alpha$ and $\gamma_s(\alpha) = \alpha \beta_{s-1}$.  By induction
\begin{equation}
 S_{\mathbb{F}_q}\left(T_{2,3,\cdots,k}(n)+\sum_{s=1}^{j+1}\gamma_s(\alpha)\prod_{l=0}^{k-s-1}X_{n-l}\right)=S_{\mathbb{F}_q}\left(T_{2,3,\cdots,k}(n)+\sum_{s=1}^{j+1}\prod_{l=0}^{k-s-1}X_{n-l}\right).
\end{equation}
Therefore,
\begin{eqnarray}\nonumber
\label{indstep1}
  S_{\mathbb{F}_q}\left(T_{2,3,\cdots,k}(n+1)+\sum_{s=1}^j\beta_s\prod_{l=0}^{k-s-1}X_{n+1-l}\right)&=& S_{\mathbb{F}_q}\left(T_{2,3,\cdots,k}(n)\right)\\
  &&+\sum_{\alpha\in\mathbb{F}_q^{\times}}  S_{\mathbb{F}_q}\left(T_{2,3,\cdots,k}(n)+\sum_{s=1}^{j+1}\prod_{l=0}^{k-s-1}X_{n-l}\right).
\end{eqnarray}
However, (\ref{indstep1}) does not depend on the choice of the $\beta_t$'s. It follows that
\begin{eqnarray}\nonumber
  S_{\mathbb{F}_q}\left(T_{2,3,\cdots,k}(n+1)+\sum_{s=1}^j\beta_s\prod_{l=0}^{k-s-1}X_{n+1-l}\right)&=& S_{\mathbb{F}_q}\left(T_{2,3,\cdots,k}(n+1)+\sum_{s=1}^j\prod_{l=0}^{k-s-1}X_{n+1-l}\right)
\end{eqnarray}
is true for $1\leq j\leq k-2$.

Consider now the case $j=k-1$.  Again,  letting $X_{n+1}$ run over every element of the field leads to 
\begin{align}\nonumber
  S_{\mathbb{F}_q}(T_{2,3,\cdots,k}(n+1)&+\sum_{s=1}^{k-1}\beta_s\prod_{l=0}^{k-s-1}X_{n+1-l})= S_{\mathbb{F}_q}\left(T_{2,3,\cdots,k}(n)\right)\\
  &+\sum_{\alpha\in\mathbb{F}_q^{\times}} e^{\frac{2\pi i}{p} \text{Tr}_{\mathbb{F}_q/\mathbb{F}_p}(\alpha \beta_{k-1})} S_{\mathbb{F}_q}\left(T_{2,3,\cdots,k}(n)+\sum_{s=1}^{k-1}\gamma_s(\alpha)\prod_{l=0}^{k-s-1}X_{n-l}\right),
\end{align}
where $\gamma_1(\alpha)=\alpha$ and $\gamma_s(\alpha) = \alpha \beta_{s-1}$.  However, by induction
\begin{equation}
 S_{\mathbb{F}_q}\left(T_{2,3,\cdots,k}(n)+\sum_{s=1}^{k-1}\gamma_s(\alpha)\prod_{l=0}^{k-s-1}X_{n-l}\right)=S_{\mathbb{F}_q}\left(T_{2,3,\cdots,k}(n)+\sum_{s=1}^{j+1}\prod_{l=0}^{k-s-1}X_{n-l}\right).
\end{equation}
Since
\begin{equation}
 \sum_{\alpha \in \mathbb{F}_q^\times}e^{\frac{2\pi i}{p} \text{Tr}_{\mathbb{F}_q/\mathbb{F}_p}(\alpha \beta_{k-1})}=-1,
\end{equation}
then it follows that
\begin{eqnarray}\nonumber
\label{indstep2}
  S_{\mathbb{F}_q}\left(T_{2,3,\cdots,k}(n+1)+\sum_{s=1}^{k-1}\beta_s\prod_{l=0}^{k-s-1}X_{n+1-l}\right)&=& S_{\mathbb{F}_q}\left(T_{2,3,\cdots,k}(n)\right)\\
  &&-\sum_{\alpha\in\mathbb{F}_q^{\times}}  S_{\mathbb{F}_q}\left(T_{2,3,\cdots,k}(n)+\sum_{s=1}^{k-1}\prod_{l=0}^{k-s-1}X_{n-l}\right).
\end{eqnarray}
Since (\ref{indstep1}) does not depend on the choice of the $\beta_t$'s, then it follows that
\begin{eqnarray}\nonumber
  S_{\mathbb{F}_q}\left(T_{2,3,\cdots,k}(n+1)+\sum_{s=1}^{k-1}\beta_s\prod_{l=0}^{k-s-1}X_{n+1-l}\right)&=& S_{\mathbb{F}_q}\left(T_{2,3,\cdots,k}(n+1)+\sum_{s=1}^{k-1}\prod_{l=0}^{k-s-1}X_{n+1-l}\right)
\end{eqnarray}
is true.  This completes the induction and the proof.
\end{proof}

Next is the recurrence for exponential sums of trapezoid functions over any Galois field.

\begin{theorem}
\label{thmtrapgen}
Let $k\geq 2$ be an integer and $q=p^r$ with $p$ prime. The sequence $\{S_{\mathbb{F}_q}(T_{2,3,\cdots,k}(n))\}_{n=k}^\infty$ satisfies a homogeneous linear recurrence with integer coefficients whose characteristic polynomial is given by 
 \begin{equation}
 \label{charpolygeneral}
  Q_{T,k,\mathbb{F}_q}(X)=X^k-q \sum _{l=0}^{k-2} (q-1)^l X^{k-2-l}.
 \end{equation}
 In particular, when $q=2$ we recover Theorem \ref{trapezoidTHM}.
\end{theorem}

\begin{proof}
We present the proof for $k>2$.  The case $k=2$ can be proved using similar techniques.
 Start by turning $X_n$ {\it OFF} and {\it ON}, that is, by letting $X_n$ assume all its possible values.  This produces the identity
 \begin{eqnarray}
 \label{firstidgeneral}
  S_{\mathbb{F}_q}(T_{2,3,\cdots,k}(n))=S_{\mathbb{F}_q}(T_{2,3,\cdots,k}(n-1))+ \sum_{\beta \in \mathbb{F}_q^{\times}}S_{\mathbb{F}_q}\left(T_{2,3,\cdots,k}(n-1)+\beta \prod_{j=1}^{k-1} X_{n-j}\right)
 \end{eqnarray}
However, Lemma \ref{generallemma} implies
\begin{equation}
 S_{\mathbb{F}_q}\left(T_{2,3,\cdots,k}(n-1)+\beta \prod_{j=1}^{k-1} X_{n-j}\right)=S_{\mathbb{F}_q}\left(T_{2,3,\cdots,k}(n-1)+ \prod_{j=1}^{k-1} X_{n-j}\right)
\end{equation}
for every $\beta \in \mathbb{F}_q^{\times}$.  Therefore, (\ref{firstidgeneral}) reduces to 
 \begin{eqnarray}
 \label{secondidgeneral}
  S_{\mathbb{F}_q}(T_{2,3,\cdots,k}(n))=S_{\mathbb{F}_q}(T_{2,3,\cdots,k}(n-1))+ (q-1)S_{\mathbb{F}_q}\left(T_{2,3,\cdots,k}(n-1)+ \prod_{j=1}^{k-1} X_{n-j}\right)
 \end{eqnarray}
Consider now $S_{\mathbb{F}_q}\left(T_{2,3,\cdots,k}(n-1)+ \prod_{j=1}^{k-1} X_{n-j}\right)$. Let $X_{n-1}$ assume all its possible values and use the same argument as before to get
\begin{eqnarray}\nonumber
  S_{\mathbb{F}_q}\left(T_{2,3,\cdots,k}(n-1)+ \prod_{j=1}^{k-1} X_{n-j}\right)&=& S_{\mathbb{F}_p}\left(T_{2,3,\cdots,k}(n-2)\right)\\
 &&+ (q-1) S_{\mathbb{F}_q}\left(T_{2,3,\cdots,k}(n-2)+\prod_{j=1}^{k-2} X_{n-1-j}+\prod_{j=1}^{k-1} X_{n-1-j}\right)
\end{eqnarray}
Thus, (\ref{secondidgeneral}) reduces to
\begin{eqnarray}\nonumber
 S_{\mathbb{F}_q}(T_{2,3,\cdots,k}(n))&=&S_{\mathbb{F}_q}(T_{2,3,\cdots,k}(n-1))+(q-1)S_{\mathbb{F}_q}\left(T_{2,3,\cdots,k}(n-2)\right)\\
 &&+ (q-1)^2 S_{\mathbb{F}_q}\left(T_{2,3,\cdots,k}(n-2)+\prod_{j=1}^{k-2} X_{n-1-j}+\prod_{j=1}^{k-1} X_{n-1-j}\right).
\end{eqnarray}
Continue in this manner to get the following equation
\begin{eqnarray}
 \label{finalidgeneral}
 S_{\mathbb{F}_q}(T_{2,3,\cdots,k}(n))&=& \sum_{l=1}^{k-1}(q-1)^{l-1}S_{\mathbb{F}_q}T_{2,3,\cdots,k}(n-l))\\\nonumber
 &&+(q-1)^{k-1} S_{\mathbb{F}_q}\left(T_{2,3,\cdots,k}(n-k+1)+\sum_{j=0}^{k-2}\prod_{l=0}^j X_{n-k+1-l}\right).
\end{eqnarray}

On the other hand, let $X_{n+1}$ assume all its possible values and use Lemma \ref{generallemma} to get the equation
\begin{eqnarray}\nonumber
\label{axuliaryfinalgen}
S_{\mathbb{F}_q}\left(T_{2,3,\cdots,k}(n+1)+\sum_{j=0}^{k-2}\prod_{i=0}^j X_{n+1-l}\right)&=& S_{\mathbb{F}_q}(T_{2,3,\cdots,k}(n))\\
&&+e^{\frac{2\pi i}{p}\text{Tr}_{\mathbb{F}_q/\mathbb{F}_p}(1)}S_{\mathbb{F}_q}\left(T_{2,3,\cdots,k}(n)+\sum_{j=0}^{k-2}\prod_{i=0}^j X_{n-l}\right)\\\nonumber
&&+e^{\frac{2\pi i}{p}\text{Tr}_{\mathbb{F}_q/\mathbb{F}_p}(2)}S_{\mathbb{F}_q}\left(T_{2,3,\cdots,k}(n)+\sum_{j=0}^{k-2}\prod_{i=0}^j X_{n-l}\right)\\\nonumber
&&+e^{\frac{2\pi i}{p}\text{Tr}_{\mathbb{F}_q/\mathbb{F}_p}(3)}S_{\mathbb{F}_q}\left(T_{2,3,\cdots,k}(n)+\sum_{j=0}^{k-2}\prod_{i=0}^j X_{n-l}\right)\\\nonumber
&&\vdots\\\nonumber
&&+e^{\frac{2\pi i}{p}\text{Tr}_{\mathbb{F}_q/\mathbb{F}_p}(\alpha_{p-1})}S_{\mathbb{F}_q}\left(T_{2,3,\cdots,k}(n)+\sum_{j=0}^{k-2}\prod_{i=0}^j X_{n-l}\right).
\end{eqnarray}
Use the well-known formula
\begin{equation}
\sum_{\beta\in\mathbb{F}_q^{\times}} e^{\frac{2\pi i}{p}\text{Tr}_{\mathbb{F}_q/\mathbb{F}_p}(\beta)}=-1.
\end{equation}  
to reduce (\ref{axuliaryfinalgen}) to 
\begin{eqnarray}
 S_{\mathbb{F}_q}\left(T_{2,3,\cdots,k}(n+1)+\sum_{j=0}^{k-2}\prod_{i=0}^j X_{n+1-l}\right)&=& S_{\mathbb{F}_q}(T_{2,3,\cdots,k}(n))\\\nonumber
 &=& - S_{\mathbb{F}_q}\left(T_{2,3,\cdots,k}(n)+\sum_{j=0}^{k-2}\prod_{i=0}^j X_{n-l}\right).
\end{eqnarray}
This last equation is equivalent to 
\begin{eqnarray}
\label{axuliaryfinalgen2}
 S_{\mathbb{F}_q}(T_{2,3,\cdots,k}(n))&=& S_{\mathbb{F}_q}\left(T_{2,3,\cdots,k}(n+1)+\sum_{j=0}^{k-2}\prod_{i=0}^j X_{n+1-l}\right)\\\nonumber
 && +S_{\mathbb{F}_q}\left(T_{2,3,\cdots,k}(n)+\sum_{j=0}^{k-2}\prod_{i=0}^j X_{n-l}\right).
\end{eqnarray}

Let $a_n = S_{\mathbb{F}_q}\left(T_{2,3,\cdots,k}(n)+\sum_{j=0}^{k-2}\prod_{i=0}^j X_{n-l}\right)$. Then,
\begin{equation}
 S_{\mathbb{F}_q}(T_{2,3,\cdots,k}(n))=a_{n+1}+a_n
\end{equation}
and equation (\ref{finalidgeneral}) is now
\begin{eqnarray}
 (a_{n+1}+a_n)&=&\sum_{l=1}^{k-1}(q-1)^{l-1}(a_{n+1-l}+a_{n-l})+(q-1)^{k-1} a_{n-k+1}.
\end{eqnarray}
The last equation reduces to
\begin{equation}
 a_{n+1}=\sum_{l=0}^{k-2} q(q-1)^{l}a_{n-1-l}
\end{equation}
This concludes the proof.
\end{proof}

The polynomial $Q_{T,k,\mathbb{F}_q}(X)$ is quite interesting. In particular, it seems to be irreducible for $k > 2$ and every $q=p^r$ with $p$ prime.  The irreducibility of $Q_{T,k,\mathbb{F}_q}(X)$ when $\gcd(k,r)=1$ 
is a consequence of Eisenstein-Dumas criterion.

\begin{theorem}[Eisenstein-Dumas criterion]
Let $f(x)=a_n x^n+a_{n-1}x^{n-1}+\cdots+a_1 x+a_0 \in \mathbb{Z}[x]$ be a polynomial. Let $p$ be a prime. Denote the $p$-adic valuation of an integer $m$ by $\nu_p(m)$ (with $\nu_p(0)=+\infty$).  Suppose that
\begin{enumerate}
 \item $\nu_p(a_n)=0$,
 \item $\nu_p(a_{n-i})/i>\nu_p(a_0)/n$ for $1\leq i \leq n-1$, and
 \item $\gcd(\nu_p(a_0),n)=1$.
\end{enumerate}
Then, $f(x)$ is irreducible over $\mathbb{Q}$.
\end{theorem}

\begin{proposition}
 Let $q=p^r$ with $p$ prime.  Suppose that $\gcd(k,r)=1$.  Then, the polynomial
 \begin{equation}
  Q_{T,k,\mathbb{F}_q}(X)=X^k-q \sum _{l=0}^{k-2} (q-1)^l X^{k-2-l}
 \end{equation}
is irreducible over $\mathbb{Q}$.
\end{proposition}

\begin{proof}
 This is a direct consequence of Eisenstein-Dumas criterion.
\end{proof}

Exponential sums over $\mathbb{F}_q$ of rotation functions also satisfy homogeneous linear recurrences.  However, in general, these linear recurrences have higher order than the 
homogeneous linear recurrences satisfied by exponential sums of trapezoid functions.  In other words, the identity observed over $\mathbb{F}_2$ between the linear recurrences of exponential sums of trapezoid Boolean functions and rotation symmetric Boolean functions is lost over $\mathbb{F}_q$.
For example, if we consider the monomial rotation 
\begin{equation}
 R_2(n)=X_1X_2+X_2X_3+\cdots+X_{n-1}X_n+X_nX_1,
\end{equation}
then we have the following result.  This is the first result that relies on linear algebra.
\begin{theorem}
\label{thmrot2}
 Suppose that $p>2$ is prime.  Then, $\{S_{\mathbb{F}_p}(R_2(n)\}$ satisfy the homogeneous linear recurrence with characteristic polynomial 
 \begin{equation}
  Q_{R,2,\mathbb{F}_p}(X)=X^{4}-p^{2}.
 \end{equation}
\end{theorem}

\begin{proof}
Turn $X_n$ and $X_{n-1}$ {\it OFF} and {\it ON}, that is, let them assume all values in $\mathbb{F}_p$, and use the identity
\begin{equation}
 S_{\mathbb{F}_p}(T_2(n)+\beta X_n) = S_{\mathbb{F}_p}(T_2(n)+ X_n), \text{ for }\beta\in\mathbb{F}_p^{\times}
\end{equation}
to get the equation 
\begin{eqnarray}
\label{rot2}
 S_{\mathbb{F}_p}(R_2(n))&=& S_{\mathbb{F}_p}(T_2(n-2))+(p-1)S_{\mathbb{F}_p}(T_2(n-2)+X_{n-2})\\\nonumber
 &&+\sum_{\alpha\in\mathbb{F}_p^\times}\sum_{\beta \in \mathbb{F}_p}e^{\frac{2\pi i}{p}\alpha\beta}S_{\mathbb{F}_p}(T_2(n-2)+\alpha X_1+\beta X_{n-2}),
\end{eqnarray}
Let 
\begin{eqnarray}
 a_0(n)&=&S_{\mathbb{F}_p}(T_2(n))\\\nonumber
 a_1(n)&=&S_{\mathbb{F}_p}(T_2(n)+X_n)\\\nonumber
 b_{\alpha,\beta}(n)&=&S_{\mathbb{F}_p}(T_2(n)+\alpha X_1+\beta X_n)\,\, \text{ for }\alpha\in\mathbb{F}_p^\times, \beta \in \mathbb{F}_p.
\end{eqnarray}
Then,
\begin{eqnarray}
 S_{\mathbb{F}_p}(R_2(n))&=& a_0(n-2)+(p-1)a_1(n-2)+\sum_{\alpha\in\mathbb{F}_p^\times}\sum_{\beta \in \mathbb{F}_p}e^{\frac{2\pi i}{p}\alpha\beta}b_{\alpha,\beta}(n-2).
\end{eqnarray}
Observe that
\begin{eqnarray}
 a_0(n)&=& a_0(n-1)+(p-1)a_1(n-1)\\\nonumber
 a_1(n)&=& a_0(n-1)-a_1(n-1)\\\nonumber
 b_{\alpha,\beta}(n)&=& \sum_{\gamma \in \mathbb{F}_p} e^{\frac{2\pi i}{p}(\beta \gamma)} b_{\alpha,\gamma}(n-1),
\end{eqnarray}
which can be written in matrix form as
\begin{equation}
 \left(
 \begin{array}{c}
 a_0(n)\\
 a_1(n)\\
 b_{1,0}(n)\\
 b_{1,1}(n)\\
 \vdots\\
 b_{p-1,p-1}(n)
 \end{array}
 \right)=A(p)\left(
 \begin{array}{c}
 a_0(n-1)\\
 a_1(n-1)\\
 b_{1,0}(n-1)\\
 b_{1,1}(n-1)\\
 \vdots\\
 b_{p-1,p-1}(n-1)
 \end{array}
 \right)
\end{equation}
where
\begin{eqnarray}
 A(p)&=\left(
\begin{array}{c|c|c|c|c}
 A_0(p)& O & O  &\cdots & O\\
 \hline
 O & A_1(p) & O &\cdots & O\\
 \hline
  O & O & A_2(p) &\cdots & O\\
  \hline
\vdots & \vdots & \vdots &\ddots & \vdots\\
\hline
  O & O & O &\cdots & A_{p-1}(p)\\
\end{array}
\right),
\end{eqnarray}
and
\begin{equation}
A_0(p)=\left(
\begin{array}{cc}
1 & p-1\\
 1 & -1 \\
 \end{array}\right)\,\,\,\, \text{ and } \,\,\,\,A_{j}(p) = \left(
\begin{array}{ccccc}
 1 & 1 & 1 & \cdots & 1 \\
 1 & e^{\frac{2  \pi i}{p}} & e^{\frac{4  \pi i}{p}}& \cdots & e^{\frac{2(p-1)\pi i}{p}} \\
  1 & e^{\frac{4   \pi i}{p}} & e^{\frac{8 \pi i}{p}}& \cdots & e^{\frac{2\times 2(p-1)\pi i}{p}} \\
 \vdots & \vdots & \vdots&  \ddots & \vdots \\
  1 & e^{\frac{2(p-1) \pi i}{p}} & e^{\frac{4(p-1) \pi i}{p}}& \cdots & e^{\frac{2\times (p-1)^2\pi i}{p}} \\
\end{array}
\right), 
\end{equation}
for $1\leq j \leq p-1.$
It is clear that the first block $A_0(p)$ satisfies $X^2-p$.  All other blocks $A_j(p)$'s, for $1\leq j \leq p-1$, are $\sqrt{p}\cdot W_p$, where $W_p$ is the $p\times p$ square Discrete Fourier 
Transform matrix.
Observe that
\begin{equation}
 A_j(p)^2 = \left(
\begin{array}{ccccc}
 p & 0 & \cdots & 0 & 0 \\
 0 & 0 & \cdots & 0 & p \\
 0 & 0 & \cdots & p & 0 \\
 \vdots & \vdots & \Ddots & \vdots & \vdots \\
 0 & p & \cdots & 0 & 0 \\
\end{array}
\right).
\end{equation}
Therefore,
\begin{equation}
A_j(p)^4 = \left(
\begin{array}{ccccc}
 p^2 & 0 & 0 &  \cdots & 0 \\
 0 & p^2 & 0 & \cdots & 0 \\
 0 & 0 & p^2 &\cdots & 0 \\
 \vdots & \vdots & \vdots & \ddots & \vdots \\
 0 & 0 & 0 &\cdots  & p^2 \\
\end{array}
\right).
\end{equation}
In other words, the big blocks $A_j(p)$'s satisfiy $X^4-p^2$.  Since $X^2-p\,|\, X^4-p^2$, then we conclude that the matrix $A(p)$ satisfies the polynomial
 \begin{equation}
  Q_{R,2,\mathbb{F}_p}(X)=X^{4}-p^{2}. 
 \end{equation}
This means that the sequences $\{a_0(n)\}$, $\{a_1(n)\}$ and $\{b_{\alpha,\beta}(n)\}$, for $\alpha \in\mathbb{F}_p^\times,\beta \in \mathbb{F}_p$, all satisfy the linear
recurrence with characteristic polynomial given by $Q_{R,2,\mathbb{F}_p}(X)$.  Since $\{S_{\mathbb{F}_p}(R_2(n))\}$ is a combination of these sequences, then it also satisfies such recurrence.
This concludes the proof.
\end{proof}

We are now ready to prove one of the main results of this article.  That is, exponential sums of rotation polynomials satisfiy linear recurrences with integer coefficients.

\begin{theorem}
\label{generalTHM}
 Let $k\geq 2$ be an integer and $q=p^r$ with $p$ prime and $r\geq 1$.  The sequence $\{S_{\mathbb{F}_q}(R_{2,3,\cdots,k}(n))\}_{n\geq k}$ satisfies a linear recurrence with integer coefficients.
\end{theorem}

\begin{proof}
 Let $\zeta_p=e^{2\pi i/p}$. Consider the expression $S_{\mathbb{F}_q}(R_{2,3,\cdots, k}(n+k))$.  Let $X_{n+k}, X_{n+k-1}, \cdots, X_n$ assume all values in $\mathbb{F}_q$ and observe that $S_{\mathbb{F}_q}(R_{2,3,\cdots, k}(n+k))$
 can be written as a linear combination of expressions of the form
 \begin{equation}
 \label{generatingseqs}
  a_{\boldsymbol{\alpha};\boldsymbol{\beta}}(n)=S_{\mathbb{F}_q}\left(T_{2,3,\cdots, k}(n)+\sum _{j=1}^{k-1} \left(\alpha _j \prod _{l=1}^j X_{n+1-l}+\beta _j \prod _{l=1}^j X_l\right)\right),
 \end{equation}
 where $\boldsymbol{\alpha}=(\alpha_1,\cdots,\alpha_k) \in \mathbb{F}_q^{k-1}$ and $\boldsymbol{\beta}=(\beta_1,\cdots,\beta_k) \in \mathbb{F}_q^{k-1}$.
 However, note that for each $\boldsymbol{\alpha},\boldsymbol{\beta} \in \mathbb{F}_q^{k-1}$, we have
 \begin{equation}
 \label{lineareqs}
  a_{\boldsymbol{\alpha};\boldsymbol{\beta}}(n) = \sum_{\boldsymbol{\gamma},\boldsymbol{\lambda}\in \mathbb{F}_q^{k-1}}c_{\boldsymbol{\gamma},\boldsymbol{\lambda}}\cdot a_{\boldsymbol{\gamma},\boldsymbol{\lambda}}(n-1),
 \end{equation}
where $c_{\boldsymbol{\gamma},\boldsymbol{\lambda}} \in \mathbb{Z}[\zeta_p]$ is a cyclotomic integer.  Let $A_{2,3,\cdots, k}(q)$ be the corresponding matrix for the linear equations in (\ref{lineareqs}) and
$F(X)$ be any annihilating polynomial for $A_{2,3,\cdots, k}(q)$.   We can assume that $F(X)$ has integer coefficients.  This is because the minimal polynomial of $A_{2,3,\cdots,k}(q)$ is monic, 
has algebraic integers coefficients and integrality is transitive.  Then each $\{a_{\boldsymbol{\alpha};\boldsymbol{\beta}}(n)\}_n$ satisfies the linear recurrence with characteristic polynomial given
by $F(X)$. Since $\{S_{\mathbb{F}_q}(R_{2,3,\cdots, k}(n+k))\}$ is a linear combination of these sequences, then $\{S_{\mathbb{F}_q}(R_{2,3,\cdots, k}(n+k))\}$ also satisfies such recurrence.
This concludes the proof.
\end{proof}

We know that the identity between linear recurrences of exponential sums of trapezoid Boolean functions and rotation symmetric Boolean functions is lost over $\mathbb{F}_q$.  However, 
the proof of Theorem \ref{generalTHM} suggests that a relation can be recovered.

\begin{corollary}
Let $q=p^r$ with $p$ prime and $r\geq 1$.  Let $\mu_{T,k,\mathbb{F}_q}(X)$ and $\mu_{R,k,\mathbb{F}_q}(X)$ be the characteristic polynomials associated to the minimal homogeneous linear recurrences with integer coefficients 
satisfied by $\{S_{\mathbb{F}_q}(T_{2,3,\cdots, k}(n))\}$ and $\{S_{\mathbb{F}_q}(R_{2,3,\cdots, k}(n))\}$ (resp.).  Then,
\begin{equation}
\label{divisibilifyres}
 \mu_{T,k,\mathbb{F}_q}(X)\,|\, \mu_{R,k,\mathbb{F}_q}(X).
\end{equation}
In particular, if $\gcd(k,r)=1$,  then $Q_{T,k,\mathbb{F}_q}(X)\,|\, \mu_{R,k,\mathbb{F}_q}(X)$
\end{corollary}

\begin{proof}
 In the proof of Theorem \ref{generalTHM}.  Observe that $\{a_{\boldsymbol{0};\boldsymbol{0}}(n)\}=\{S_{q}(T_{2,3,\cdots,k}(n))\}$, this implies (\ref{divisibilifyres}).  Now, if $\gcd(k,r)=1$,  
then $Q_{T,k,\mathbb{F}_q}(X)$ is irreducible and therefore $\mu_{T,k,\mathbb{F}_q}(X)=Q_{T,k,\mathbb{F}_q}(X)$.  This concludes the proof.
\end{proof}

\begin{definition}
Let $\{b(n)\}$ be a sequence on an integral domain $D$.  A set of sequences $$\{\{a_1(n)\}, \{a_2(n)\},\cdots, \{a_s(n)\}\},$$ where $s$ is some natural number, is called a 
{\em recursive generating set for} $\{b(n)\}$ if 
\begin{enumerate}
 \item there is an integer $l$ such that for every $n$, $b(n)$ can be written as a linear combination of the form 
 $$b(n)=\sum_{j=1}^s c_j\cdot a_j(n-l),$$  
 where $c_j$'s are constants that belong to $D$, and
 \item for each $1\leq j_0 \leq s$ and every $n$, $a_{j_0}(n)$ can be written as a linear combination of the form 
 $$a_{j_0}(n)=\sum_{j=1}^s d_j\cdot a_j(n-1),$$
 where $d_j$'s are also constants that belong to $D$.
\end{enumerate}
The sequences $\{a_j(n)\}$'s are called {\it recursive generating sequences for} $\{b(n)\}$.
\end{definition}

\begin{remark}
It is a well-known result in the theory of recursive sequences that a sequence that has a recursive generating set satisfies a linear recurrence with constant coefficients.  In fact, this technique 
has been used in Theorems \ref{thmrot2} and \ref{generalTHM}.
\end{remark}

Theorem \ref{generalTHM} generalizes to monomial rotation functions and linear combinations of them, that is, exponential sums over any Galois field of linear combinations of monomial rotation polynomials 
satisfy linear recurrences.  Of course, in general, we might need to turn {\it OFF} and {\it ON} more than $k$ variables, even if the rotation is of degree $k$.  Also, even though the sequences 
(\ref{generatingseqs}) always exist, their number might be too big to be handled by hand.  For example, consider the sequence of exponential sums $\{S_{\mathbb{F}_3}(R_{2,3}(n))\}$.  After some 
identifications, the authors needed 24 different recursive generating sequences (not claiming that this is optimal) of the form  (\ref{generatingseqs}) and their corresponding $24\times 24$ matrix in order 
to find that $\{S_{\mathbb{F}_3}(R_{2,3}(n))\}$ satisfy the linear recurrence whose characteristic polynomial is given by
\begin{eqnarray}
 X^6-3 X^4-9 X^3+9 X+18&=&\left(X^3-3\right) \left(X^3-3 X-6\right)\\\nonumber
 &=&\left(X^3-3\right)Q_{T,3,\mathbb{F}_3}(X).
\end{eqnarray}
Also, in general, finding the minimal polynomial of a matrix is not an easy task, therefore explicit formulas like the ones in Theorem \ref{thmtrapgen} and Theorem \ref{thmrot2} are much harder to get.

In the next section, this technique is used to prove that exponential sums over Galois fields of elementary symmetric polynomials (and linear combinations of them) satisfy homogeneous linear
recurrences with integer coefficients.


\section{Linear recurrences over $\mathbb{F}_q$: Symmetric polynomials case}
\label{symmetriccase}

It is a well-established result that exponential sums of symmetric Boolean functions are linear recurrent.  This was first established by Cai, Green and Thierauf \cite{cai}.   In \cite{cm1}, Castro and 
Medina use this result to show that a conjecture of Cusick, Li, St$\check{\mbox{a}}$nic$\check{\mbox{a}}$ \cite{cusick2} is true asymptotically.  In \cite{cm2}, some of the results of \cite{cm1} where 
extended to some perturbations of symmetric Boolean functions.  This recursivity was also used in \cite{cm3, cusick4} to study the periodicity mod $p$ ($p$ prime) of exponential sums of symmetric Boolean functions. 

In this section we show that exponential sums of some symmetric polynomials are linear recurrent over any Galois field.  Remarkably, the proof uses the same argument as in the proof of 
Theorem \ref{generalTHM}. We decided to include the proof for completeness of the writing.  However, the reader is welcome to skip the proof.

Let $\sigma_{n,k}$ be the elementary symmetric polynomial in $n$ variables of degree $k$. For example,
\begin{equation}
\sigma_{4,3} = X_1 X_2 X_3+X_1 X_4 X_3+X_2 X_4 X_3+X_1 X_2 X_4.
\end{equation}
We have the following result.

\begin{theorem}
\label{generalTHMSymmetric}
 Let $k\geq 2$ be an integer and $q=p^r$ with $p$ prime and $r\geq 1$.  The sequence $\{S_{\mathbb{F}_q}(\sigma_{n,k})\}$ satisfies a linear recurrence with constant coefficients.
\end{theorem}

\begin{proof}
Consider the expression $S_{\mathbb{F}_q}(\sigma_{n+k,k})$.  Define
\begin{equation}
 \label{generatingseqsSym}
  a_{\boldsymbol{\beta}}(n)=S_{\mathbb{F}_q}\left(\sigma_{n,k}+\sum _{j=1}^{k-1} \beta_j \sigma_{n,k-j} \right),
 \end{equation}
The set $\{a_{\boldsymbol{\beta}}(n)\}_{\boldsymbol{\beta}\in \mathbb{F}_q^{k-1}}$ is a recursive generating set for $S_{\mathbb{F}_q}(\sigma_{n+k,k})$.  Therefore, the sequence 
$\{S_{\mathbb{F}_q}(\sigma_{n+k,k})\}_{n\geq 0}$ satisfies a linear recurrence with constant coefficients.  As in the proof of Theorem \ref{generalTHM}, it can be argued that a linear recurrence with 
integer coefficients is guaranteed to exist.  This concludes the proof.
\end{proof}

This result can be generalized to any polynomial of the form 
\begin{equation}
\sum _{j=0}^{k-1} \beta_j \sigma_{n,k-j},
\end{equation}
with $\beta_j \in \mathbb{F}_q$.  We present the result without proof, as it follows almost verbatim as the one from Theorem \ref{generalTHMSymmetric}.

\begin{theorem}
\label{MoregeneralTHMSymmetric}
 Let $k\geq 2$ be an integer and $q=p^r$ with $p$ prime and $r\geq 1$.  The sequence 
\begin{equation}
 S_{\mathbb{F}_q}\left(\sum _{j=0}^{k-1} \beta_j \sigma_{n,k-j}\right)
\end{equation}
satisfies a linear recurrence with constant coefficients, regardless of the choice of the $\beta_j$'s.
\end{theorem}

\begin{example}
 Consider the sequence $\{S_{\mathbb{F}_3}(\sigma_{n,3})\}$.  Recall that in this case the generating sequences are given by 
 \begin{equation}
  a_{(s,t)}(n)=\{S_{\mathbb{F}_3}(\sigma_{n,3}+s\sigma_{n,2}+t\sigma_{n,1})\},
 \end{equation}
 where $s,t\in \mathbb{F}_3$.  Establish the order
 $$(0,0), (1,0), (2,0), (0,1), (1,1), (2,1), (0,2),(1,2), (2,2).$$
 Then, 
 \begin{equation}
\left( \begin{array}{c}
  a_{(0,0)}(n)\\
  a_{(1,0)}(n)\\
  \vdots\\
  a_{(2,2)}(n)
 \end{array}
\right) = A \left( \begin{array}{c}
  a_{(0,0)}(n-1)\\
  a_{(1,0)}(n-1)\\
  \vdots\\
  a_{(2,2)}(n-1)
 \end{array}
\right),
\end{equation}
where the matrix $A$ is given by
\begin{equation}
 A=\left(
\begin{array}{ccccccccc}
 1 & 1 & 1 & 0 & 0 & 0 & 0 & 0 & 0 \\
 0 & 1 & 0 & 0 & 0 & 1 & 1 & 0 & 0 \\
 0 & 0 & 1 & 0 & 1 & 0 & 1 & 0 & 0 \\
 0 & 0 & 0 & 1 & e^{\frac{2 i \pi }{3}} & e^{-\frac{2 i \pi }{3}} & 0 & 0 & 0 \\
 e^{-\frac{2 i \pi }{3}} & 0 & 0 & 0 & 1 & 0 & 0 & 0 & e^{\frac{2 i \pi }{3}} \\
 e^{\frac{2 i \pi }{3}} & 0 & 0 & 0 & 0 & 1 & 0 & e^{-\frac{2 i \pi }{3}} & 0 \\
 0 & 0 & 0 & 0 & 0 & 0 & 1 & e^{-\frac{2 i \pi }{3}} & e^{\frac{2 i \pi }{3}} \\
 0 & 0 & e^{-\frac{2 i \pi }{3}} & e^{\frac{2 i \pi }{3}} & 0 & 0 & 0 & 1 & 0 \\
 0 & e^{\frac{2 i \pi }{3}} & 0 & e^{-\frac{2 i \pi }{3}} & 0 & 0 & 0 & 0 & 1 \\
\end{array}
\right).
\end{equation}
The minimal polynomial of $A$ is given by 
\begin{eqnarray}
 \mu_A(X) &=& X^9-9 X^8+36 X^7-81 X^6+108 X^5-81 X^4+81 X^2-81 X+27\\\nonumber
 &=& \left(X^3-3 X^2+3\right) \left(X^6-6 X^5+18 X^4-30 X^3+36 X^2-27 X+9\right).
\end{eqnarray}
Therefore, $\{S_{\mathbb{F}_3}(\sigma_{n,3})\}$ satisfies the linear recurrence with characteristic polynomial given by $\mu_A(X)$.
\end{example}

\subsection{Quadratic case}

The case of the elementary symmetric polynomial of degree 2 is fascinating.  Observe that
 \begin{equation}
  a_{s}(n) = S_{\mathbb{F}_p}(\sigma_{n,2}+s \sigma_{n,1}),
 \end{equation}
where $s\in \mathbb{F}_p$, are the generating sequences of $\{S_{\mathbb{F}_p}(\sigma_{n,2})\}$.  Also,
\begin{equation}
\left( \begin{array}{c}
  a_0(n)\\
  a_1(n)\\
  \vdots\\
  a_{p-1}(n)
 \end{array}
\right) = M(p) \left( \begin{array}{c}
  a_0(n-1)\\
  a_1(n-1)\\
  \vdots\\
  a_{p-1}(n-1)
 \end{array}
\right),
\end{equation}
where the matrix $M(p)$ is given by
\begin{equation}
M(p)=\left(\begin{array}{ccccccc}
1 & 1 & 1 & 1& 1 & \cdots & 1 \\
e^{\frac{2(p-1)\pi i}{p}} & 1 & e^{\frac{2 \pi i}{p}} & e^{\frac{4 \pi i}{p}} & e^{\frac{6 \pi i}{p}}& \cdots & e^{\frac{2(p-2)\pi i}{p}}\\
e^{\frac{2\times 2(p-2)\pi i}{p}}& e^{\frac{2\times 2(p-1)\pi i}{p}} & 1 & e^{\frac{4  \pi i}{p}} & e^{\frac{8 \pi i}{p}}& \cdots & e^{\frac{2\times 2(p-3)\pi i}{p}} \\
e^{\frac{2\times 3(p-3)\pi i}{p}}& e^{\frac{2\times 3(p-2)\pi i}{p}} & e^{\frac{2\times 3(p-1)\pi i}{p}} & 1 & e^{\frac{6  \pi i}{p}} &  \cdots & e^{\frac{2\times 2(p-3)\pi i}{p}} \\
e^{\frac{2\times 4(p-4)\pi i}{p}}& e^{\frac{2\times 4(p-3)\pi i}{p}} & e^{\frac{2\times 4(p-2)\pi i}{p}} & e^{\frac{2\times 4(p-2)\pi i}{p}} &  1 &   \cdots & e^{\frac{2\times 2(p-3)\pi i}{p}} \\
 \vdots& \vdots& \vdots & \vdots & \vdots& \ddots & \vdots \\
e^{\frac{2(p-1) \pi i}{p}} & e^{\frac{2\times 2(p-1) \pi i}{p}}  & e^{\frac{2\times 3(p-1) \pi i}{p}} & e^{\frac{2\times 4(p-1) \pi i}{p}} & e^{\frac{2\times 5 (p-1)\pi i}{p}}  & \cdots & 1 \\
 \end{array}\right).
\end{equation}
The matrix $M(p)$ can be obtained from the $p\times p$ Fourier Discrete Transform Matrix by replacing its $j$-row ${\bf r}_j$ by $RTC^{j-1}({\bf r}_j)$, where $RTC$ is the {\it rotate through carry}
function
\begin{equation}
 RTC(a_1,a_2,a_3,\cdots, a_n) = (a_n,a_1,a_2,\cdots, a_{n-1})
\end{equation}
and $RTC^m$ represents $m$ iterations of $RTC$.

It is not hard to prove that $M(p)$ is a Complex Hadamard Matrix.  In particular,
\begin{equation}
M(p) \overline{M(p)}^T = \overline{M(p)}^T M(p) = \left(
\begin{array}{ccccc}
  p & 0 & 0 & \cdots & 0 \\
 0 & p & 0 & \cdots & 0 \\
 0 & 0 &p & \cdots & 0 \\
 \vdots & \vdots & \vdots & \ddots  & \vdots \\
 0 & 0 & 0 & \cdots & p \\
\end{array}
\right).
\end{equation}
This implies that $M(p)$ is diagonalizable and that all its eigenvalues satisfy $|\lambda|=\sqrt{p}$.  Moreover, its eigenvalues are related to the number-theoretical quadratic Gauss sum mod $p$.  
The {\it quadratic Gauss sum mod $p$} is defined by
\begin{equation}
g(a;p)=\sum_{k=0}^{p-1}e^{2\pi i ak^2/p}
\end{equation}
It is well-established that
\begin{equation}
g(a;p) = \left(\frac{a}{p}\right)g(1;p),
\end{equation}
where $(a/p)$ denotes the Legendre's symbol, and that
\begin{equation}
\label{gaussvalue}
g(1;p) = \begin{cases}
 \sqrt{p} & p\equiv 1 \mod 4 \\
 i\sqrt{p} & p\equiv 3 \mod 4.
\end{cases}
\end{equation}

\begin{theorem}
\label{eigenthm}
Let $C(p)$ be the set of eigenvalues of $M(p)$. Let $\zeta_p =e^{2\pi i/p}$.  Then, $\lambda \in C(p)$ if and only if
\begin{equation}
\label{eigenvalues}
\lambda = \leg{-2}p g(1;p)\zeta^{-s a^2}.
\end{equation}
In particular, $|C(p)|=(p+1)/2$.
\end{theorem}

\begin{proof}
Let $p$ be an odd prime number and $\zeta=\exp(2\pi i/p)$. The matrix
$M(p)$ has $(j,k)$-entry $\zeta^{j(k-j)}$ where $j$ and
$k$ run from $0$ to $p-1$ inclusive.  We compute the eigenvalues
of $M(p)$ simply by writing down its eigenvectors.

Set $s=\frac12(p-1)$. Then $1\equiv-2s$ (mod~$p$)
For $0\le a\le p-1$, let $v_a$ be the column vector with $k$-entry
$\zeta^{s(k-a)^2}$ where $0\le k\le p-1$. Then the $v_a$ are the cyclic shifts of
$v_0$. The entry in row $j$ of $M(p)v_a$ is
\begin{eqnarray*}
\sum_{k=0}^{p-1}\zeta^{j(k-j)+s(k-a)^2}
&=&\sum_{k=0}^{p-1}\zeta^{-2sjk+2sj^2+sk^2-2sak+sa^2}\\
&=&\sum_{k=0}^{p-1}\zeta^{s(k-a-j)^2+sj^2-2saj}\\
&=&g(s;p)\zeta^{s(j-a)^2-sa^2}.
\end{eqnarray*}
This is $g(s,p)\zeta^{-sa^2}$ times the entry in row $j$ of $v_a$.
Therefore each $v_a$ is an eigenvector with eigenvalue
$$g(s;p)\zeta^{-sa^2}=\leg sp g(1;p)\zeta^{-sa^2}
=\leg {-2}p g(1;p)\zeta^{-sa^2}.$$

As these eigenvalues are not all distinct, there remains the possibility
that some of these eigenvectors $v_a$ are not linearly independent.
That can only happen with eigenvectors in the same eigenspace, so
for $v_a$ and $v_{p-a}$ where $0<a<p$. But it is clear that none
of the $v_a$ are multiples of any of the others; simply consider
the quotients of corresponding entries. So we have a dimension-two
eigenspace for each eigenvalue $\leg{-2}p g(1,p)\zeta^{-sa^2}$
for $1\le a\le\frac12(p-1)$. This completes the proof.
\end{proof}

Note that if $\lambda$ is defined as in (\ref{eigenvalues}), then equation (\ref{gaussvalue}) implies
\begin{equation}
 \lambda^p = (-i)^\frac{p-1}{2}\sqrt{p^p}
\end{equation}
for every odd prime $p$.  Therefore, Theorem \ref{eigenthm} leads to
\begin{equation}
 M(p)^p =\left(
\begin{array}{ccccc}
  (-i)^{\frac{p-1}{2}}\sqrt{p^p} & 0 & 0 & \cdots & 0 \\
 0 & (-i)^{\frac{p-1}{2}}\sqrt{p^p} & 0 & \cdots & 0 \\
 0 & 0 &(-i)^{\frac{p-1}{2}}\sqrt{p^p} & \cdots & 0 \\
 \vdots & \vdots & \vdots & \ddots  & \vdots \\
 0 & 0 & 0 & \cdots & (-i)^{\frac{p-1}{2}}\sqrt{p^p} \\
\end{array}
\right),
\end{equation}
and so
\begin{equation}
 M(p)^{2p} =\left(
\begin{array}{ccccc}
  \left(\frac{-1}{p}\right)p^p & 0 & 0 & \cdots & 0 \\
 0 & \left(\frac{-1}{p}\right)p^p & 0 & \cdots & 0 \\
 0 & 0 &\left(\frac{-1}{p}\right)p^p & \cdots & 0 \\
 \vdots & \vdots & \vdots & \ddots  & \vdots \\
 0 & 0 & 0 & \cdots & \left(\frac{-1}{p}\right)p^p \\
\end{array}
\right).
\end{equation}
Thus,
\begin{equation}
\label{chark2}
 X^{2p}-\left(\frac{-1}{p}\right)p^p
\end{equation}
is an annihilating polynomial for the matrix $M(p)$, which in turns implies that $\{S_{\mathbb{F}_p}(\sigma_{n,2})\}$ satisfies the linear recurrence with characteristic polynomial (\ref{chark2}).

\section{Some observations and concluding remarks}

We had shown that exponential sums over Galois fields of trapezoid polynomials and rotation polynomials satisfy linear recurrences with integer coefficients.  This means that they can be
calculated efficiently if we know a priori some initial values. We predict the initial conditions for two families of these type of polynomials.

Consider the trapezoid polynomial $T_{2,3,\cdots,k}(n)$. Recall that $\{S_{\mathbb{F}_q}(T_{2,3,\cdots,k}(n))\}$ satisfies the linear recurrence with integer coefficients with characteristic polynomial 
given by $Q_{T,k,\mathbb{F}_q}(X)$, which is of degree $k$.  This implies that we need to know $k$ initial values in order to calculate the whole sequence.  Of course, $\{S_{\mathbb{F}_q}(T_{2,3,\cdots,k}(n))\}$
makes sense only for values of $n\geq k$, however, since it satisfies a linear recurrence with integer coefficients, it can be extended to values of $n<k$.  We conjecture the following.
\begin{conjecture}
\label{trapconj}
 Let $\{t_{k,q}(n)\}$ be defined by 
 \begin{eqnarray}
  t_{k,q}(j) &=& q^j,\,\,\, \text{ for }0\leq j \leq k-1\\\nonumber
  t_{k,q}(n) &=& =q \sum _{l=0}^{k-2} (q-1)^l t_{k,q}(n-(l+2)), \,\,\, \text{ for }n\geq k.
 \end{eqnarray}
Then, $S_{\mathbb{F}_q}(T_{2,3,\cdots,k}(n)) = t_{k,q}(n)$ for all values of $n\geq k$.
\end{conjecture}
We were able to prove that this conjecture is true for $k=2,3,4$, but the general statement remains open.  We were also able to predict the initial conditions for $\{S(R_{2,3,\cdots, k}(n))\}$ (Boolean case). Recall that this sequence satisfies the linear recurrence whose characteristic polynomial is given by
\begin{equation}
 p_k(X)=X^k -2 (X^{k-2}+X^{k-3}+\cdots+X+1).
\end{equation}
Therefore, as in the case of trapezoid polynomial $T_{2,3,\cdots,k}(n)$, we need to know $k$ initial values in order to calculate the whole sequence.
\begin{conjecture}
\label{rotconj}
Let
\begin{equation}
\delta_o(j) =\begin{cases}
 0 & \text{ if }j \text{ is even} \\
 1 & \text{ if }j \text{ is odd.}
\end{cases}
\end{equation}
Define $\{r_{k}(n)\}$ by 
 \begin{eqnarray}
  r_k(0)&=& k\\\nonumber
  r_k(j)&=& 2^j-\delta_o(j)\cdot 2,\,\,\, \text{ for }1\leq j\leq k-1\\\nonumber
  r_k(n)&=& 2\sum_{l=0}^{k-2} r_k(n-(l+2)), \,\,\, \text{ for }n\geq k.
 \end{eqnarray}
Then, $S(R_{2,3,\cdots,k}(n)) = r_{k}(n)$ for all values of $n\geq k$.
\end{conjecture}

The problem of finding suitable initial conditions for this type of sequences is a nice problem, but also an important one.  For example, if 
Conjecture \ref{rotconj} is true, then 
\begin{eqnarray*}
 \{S(R_{2,3,\cdots,15}(n))\}_{n\geq 15} &=& 32766, 65504, 131036, 262036, 524096, 104813,2096268, 4192412\cdots \\
 \{S(R_{2,3,\cdots,30}(n))\}_{n\geq 30}&=&1073086444, 2146129256, 4292171136, 8584167576, 17167985776,\\
 && 34335272736, 68669148016, 137335500952,\cdots\\
 \{S(R_{2,3,\cdots,100}(n))\}_{n\geq 100}&=& 1267650600228229401496703205376, 2535301200456458802993406410548,\\
&&5070602400912917605986812821300, 10141204801825835211973625642388, \\
&&20282409603651670423947251284976, 40564819207303340847894502569720, \\
&&\cdots
\end{eqnarray*}
On the other hand, we know that Conjecture \ref{trapconj} is true for $k=3$, which means, for example, that
\begin{eqnarray*}
 \{S_{\mathbb{F}_9}(T_{2,3}(n))\}_{n\geq 3} &=& 153, 1377, 7209, 23409, 164025, 729729, 3161673, 18377361,\cdots\\
 \{S_{\mathbb{F}_{7^3}}(T_{2,3}(n))\}_{n\geq 3} &=& 234955, 80589565, 13881523159, 55203852025, 14215001955427,\\
&&1647320876934229, 11351488736356111, 2232536080171760209,\cdots\\
 \{S_{\mathbb{F}_{71^2}}(T_{2,3}(n))\}_{n\geq 3} &=& 50818321, 256175156161, 645881606118001, 2582501749259041, \\
&& 9764439145967152081, 16422699840579863752321, \\
&&114835229977615135072561, 330868420079857977922668001,\cdots.
\end{eqnarray*}
Also, if Conjecture \ref{trapconj} is true in general, then we have, for example, 
\begin{eqnarray*}
 \{S_{\mathbb{F}_5}(T_{2,3,4,5}(n))\}_{n\geq 5} &=& 1845, 9225, 39725, 173025, 730725, 2988025, 13244125, 56108625,\cdots\\
 \{S_{\mathbb{F}_{11^2}}(T_{2,3,\cdots, 7}(n))\}_{n\geq 7} &=& 18445769583241, 2231938119572161, 226346720724231481,\\
&&22141818198352009201, 2044333948085969113321,\\
&& 170550498912524502711841, 11342127359186464124132761,\cdots\\
 \{S_{\mathbb{F}_{7919}}(T_{2,3,\cdots,8}(n))\}_{n\geq 8} &=& 13665512318276822315545157633, 108217192048434155916802103295727, \\
&&734609211013142008709051078210604961,\\ 
&&4848502223556916452901817857822360556623,\\
&&30722822355930196223839440343843855453844801, \\
&&182535766024343164334384388453936618605681619887,\cdots.
\end{eqnarray*}

All these values where calculated almost instantaneously.  Another nice problem is to automatize the process presented in this work.

\bibliographystyle{plain}

\begin{thebibliography}{1}
\bibitem{sperber} A. Adolphson and S. Sperber. 
\newblock $p$-adic Estimates for Exponential Sums and the of Chevalley-Warning.
\newblock {\it Ann. Sci. Ec. Norm. Super.}, $4^{e}$ s\'erie,  {\bf 20}, 545--556, 1987.

\bibitem{ax} J. Ax.
\newblock Zeros of polynomials over finite fields. 
\newblock {\it Amer. J. Math.}, {\bf 86}, 255--261, 1964.

\bibitem{BCP} M. L. Bileschi, T.W. Cusick and D. Padgett.
\newblock Weights of Boolean cubic monomial rotation symmetric functions.
\newblock {\it Cryptogr. Commun.}, {\bf 4}, 105--130, 2012.

\bibitem{cai} J. Cai, F. Green and T. Thierauf. 
\newblock On the correlation of symmetric functions.
\newblock {\it Math. Systems Theory}, {\bf 29}, 245--€"258, 1996.

\bibitem{cm1} F. N. Castro and L. A. Medina. 
\newblock Linear Recurrences and Asymptotic Behavior of Exponential Sums of Symmetric Boolean Functions. 
\newblock {\it Elec. J. Combinatorics}, 18:\#P8, 2011.

\bibitem{cm2} F. N. Castro and L. A. Medina. 
\newblock Asymptotic Behavior of Perturbations of Symmetric Functions.  
\newblock {\it Annals of Combinatorics}, 18:397--417, 2014.

\bibitem{cm3} F. N. Castro and L. A. Medina. 
\newblock Modular periodicity of exponential sums of symmetric Boolean functions.
\newblock {\it Discrete Appl. Math.} {\bf 217}, 455--473, 2017.

\bibitem{cusick4} T. W. Cusick. 
\newblock Hamming weights of symmetric Boolean functions.
\newblock {\it Discrete Appl. Math.} {\bf 215}, 14--19, 2016.

\bibitem{cusickArXiv} T. W. Cusick.
\newblock Weight recursions for any rotation symmetric Boolean functions.
\newblock arXiv:1701.06648 [math.CO]

\bibitem{cusickjohns} T. W. Cusick and B. Johns.
\newblock Recursion orders for weights of Boolean cubic rotation symmetric functions.
\newblock {\it Discr. Appl. Math.}, {\bf 186}, 1--6, 2015.

\bibitem{cusick2} T. W. Cusick, Y. Li, and  P. St$\check{\mbox{a}}$nic$\check{\mbox{a}}$.
\newblock Balanced Symmetric Functions over $GF(p)$.
\newblock {\it IEEE Trans. on Information Theory} {\bf 5}, 1304-1307, 2008.

\bibitem{cusickstanica} T.W. Cusick and P. St$\check{\mbox{a}}$nic$\check{\mbox{a}}$.
\newblock Fast evaluation, weights and nonlinearity of rotation symmetric functions.
\newblock {\it Discr. Math.}, {\bf 258}, 289--301, 2002.

\bibitem{dalaimaitrasarkar} D. K. Dalai, S. Maitra and S. Sarkar. 
\newblock Results on rotation symmetric Bent functions.
\newblock {\it Second International Workshop on Boolean Functions: Cryptography and
Applications, BFCA'06}, publications of the universities of Rouen and Havre, 137--156, 2006.

\bibitem{hell} M. Hell, A. Maximov and S. Maitra. 
\newblock On efficient implementation of search strategy for rotation symmetric Boolean functions. 
\newblock {\it Ninth International Workshop on Algebraic and Combinatorial Coding Theory, ACCT 2004}, Black Sea Coast, Bulgaria,
2004.

\bibitem{fspectrum} M. Kolountzakis, R. J. Lipton, E. Markakis, A. Metha and N. K. Vishnoi. 
\newblock On the Fourier Spectrum of Symmetric Boolean Functions.
\newblock {\it Combinatorica}, {\bf  29}, 363--387, 2009.

\bibitem{maxhellmaitra} A. Maximov, M. Hell and S. Maitra. 
\newblock Plateaued Rotation Symmetric Boolean Functions on Odd Number of Variables. 
\newblock {\it First Workshop on Boolean Functions:Cryptography and Applications, BFCA'05}, publications of the universities of Rouen and Havre, 83--104, 2005.

\bibitem{mm1} O. Moreno and C. J. Moreno.
\newblock Improvement of the Chevalley-Warning and the Ax-Katz theorems.
\newblock {\it Amer. J. Math.}, {\bf 117}, 241--244, 1995.

\bibitem{mm} O. Moreno and C. J. Moreno.
\newblock The MacWilliams-Sloane Conjecture on the Tightness of the Carlitz-Uchiyama Bound and the Weights of Dual of BCH Codes.
\newblock {\it IEEE Trans. Inform. Theory}, {\bf 40},  1894--1907, 1994.

\bibitem{piequ} J. Pieprzyk and C.X. Qu.
\newblock Fast hashing and rotation-symmetric functions.
\newblock {\it J. Universal Comput. Sci.}, {\bf 5 (1)}, 20--31, 1999.

\bibitem{fdegree} A. Shpilka and A. Tal.
\newblock On the Minimal Fourier Degree of Symmetric Boolean Functions.
\newblock {\it Combinatorica}, {\bf 88}, 359--377, 2014.

\bibitem{stanicamaitra} P. St$\check{\mbox{a}}$nic$\check{\mbox{a}}$ and S. Maitra. 
\newblock Rotation Symmetric Boolean Functions -- Count and Cryptographic Properties. 
\newblock {\it Discr. Appl. Math.}, {\bf 156}, 1567--1580, 2008

\bibitem{stanicamaitraclark} P. St$\check{\mbox{a}}$nic$\check{\mbox{a}}$, S. Maitra and J. Clark. 
\newblock Results on Rotation Symmetric Bent and Correlation Immune Boolean Functions. 
\newblock {\it Fast Software Encryption, FSE 2004}, Lecture Notes in Computer Science, {\bf 3017}, 161--177. SpringerVerlag, 2004.



\end{thebibliography}

\end{document}